\documentclass[11pt]{amsart}
\usepackage{amsmath,amstext,amssymb,amsopn,amsthm}
\usepackage{wasysym}
\usepackage[top=25mm,bottom=25mm,left=25mm,right=25mm]{geometry}

\usepackage{hyperref}
\usepackage{mathrsfs}

\newtheorem{theorem}{Theorem}[section]
\newtheorem{lemma}[theorem]{Lemma}
\newtheorem{proposition}[theorem]{Proposition}
\newtheorem{corollary}[theorem]{Corollary}
\newtheorem{definition}[theorem]{Definition}
\newtheorem{example}[theorem]{Example}
\newtheorem{remark}[theorem]{Remark}

\makeatletter
\@addtoreset{equation}{section}
\makeatother

\newcommand{\N}{\mathbb{N}}
\newcommand{\mpr}{\mathbb P}
\newcommand{\mr}{{\mathbb R}}
\newcommand{\R}{{\mathbb R}}

\newcommand{\Rd}{{\mr^d}}
\newcommand{\dx}{{\rm d}x}
\newcommand{\dy}{{\rm d}y}
\newcommand{\dz}{{\rm d}z}
\newcommand{\dt}{{\rm d}t}
\DeclareMathOperator{\dist}{dist}
\newcommand{\dw}{{\rm d}w}
\newcommand{\ds}{{\rm d}s}
\newcommand*\eps{\epsilon}
\newcommand*\olp{\overline{P}}
\newcommand*\oU{\overline{U}}
\newcommand*\pord{\int_{\mR^d}}
\newcommand*\podc{\int_{D^c}}

\newcommand*\mE{\mathbb{E}}
\newcommand*\mR{\mathbb{R}}
\newcommand*\me{\mathrm{e}}
\newcommand*\nudy{\nu(\mathrm{d}y)}
\newcommand*\pordd{\int_{\mR^{2d}}}
\usepackage{color}
\usepackage{xcolor}

\newcommand*\vd{\mathcal V^D}
\newcommand*\xd{\mathcal{X}^{D}}
\DeclareMathOperator{\diam}{diam}
\newcommand{\ED}{{\mathcal E}_D}
\newcommand{\HD}{{\mathcal H}_D}



\title[Extension and trace
for nonlocal operators]{Extension and trace
for nonlocal operators}
\author[K. Bogdan]{Krzysztof Bogdan}
\address{ Faculty of Pure and Applied Mathematics,
Wroc\l aw University of Science and Technology,
Wyb. Wyspia\'nskiego 27, 50-370 Wroc\l aw, Poland}
\email{bogdan@pwr.edu.pl}
\author[T. Grzywny]{Tomasz Grzywny}
\address{ Faculty of Pure and Applied Mathematics,
Wroc\l aw University of Science and Technology,
Wyb. Wyspia\'nskiego 27, 50-370 Wroc\l aw, Poland}
\email{tomasz.grzywny@pwr.edu.pl}
\author[K. Pietruska-Pa\l{}uba]{Katarzyna Pietruska-Pa\l{}uba}
\address{Institute of Mathematics, University of Warsaw, ul Banacha 2, 02-097 Warsaw, Poland}
\email{kpp@mimuw.edu.pl}
\author[A. Rutkowski]{Artur Rutkowski}
\address{ Faculty of Pure and Applied Mathematics,
Wroc\l aw University of Science and Technology,
Wyb. Wyspia\'nskiego 27, 50-370 Wroc\l aw, Poland}
\email{artur.rutkowski@pwr.edu.pl}
\thanks{The research was partially supported by the NCN grant 2014/14/M/ST1/00600. The fourth named author was also supported by the NCN grant 2015/18/E/ST1/00239}
\begin{document}
\begin{abstract}
	We prove an optimal extension and trace theorem for Sobolev spaces of nonlocal operators.
	The extension  is given by a suitable Poisson integral and
	solves the corresponding nonlocal Dirichlet problem. We
give  a Douglas-type formula for
the
quadratic form of the
Poisson extension.

\end{abstract}
\maketitle

\section{Introduction}\label{sec:Intro}

Let $d=1,2,\ldots$. Let $\nu\colon [0,\infty)\to (0,\infty]$ be nonincreasing and denote $\nu(z)=\nu(|z|)$ for $z\in \Rd$.
In particular, $\nu(z)=\nu(-z)$. We assume that $\int_\Rd \nu(z)\,\dz=\infty$
but
$$
\int_\Rd \left(|z|^2\wedge 1\right)\nu(z)\dz<\infty.
$$
Summarizing, $\nu$ is a strictly positive density function of
isotropic infinite unimodal  L\'{e}vy measure on $\Rd$,
in short,
$\nu$ is {\it unimodal}. We will add further assumptions on $\nu$ later on.
For $x\in \Rd$ and $u\colon \Rd\to \R$ we let
\begin{eqnarray}\label{eq:L-def}
Lu(x)
&=&\lim_{\epsilon\to 0^+} \int_{|x-y|>\epsilon} (u(y)-u(x))\nu(y-x)\,\dy\\
&=&\lim_{\epsilon\to 0^+} \tfrac12 \!\int_{|x-y|>\epsilon} \!\!\!(u(x+z)+u(x-z)-2u(x))\nu(z)\,\dz.\nonumber
\end{eqnarray}
Here and in what follows all the considered
sets, functions and measures are assumed to be Borel.
The limit in \eqref{eq:L-def} exists,
e.g., for
$u\in C_c^\infty (\mathbb R^d)$, the smooth functions with compact support.
We note that $L$ is a non-local symmetric translation-invariant linear operator on $C_c^\infty (\mathbb R^d)$ satisfying the positive maximum principle, cf. \cite[Section 2]{MR2320691}.
For example,  if  $0<\alpha<2$  and
\begin{equation}
\label{eq:lm}
\nu(z)=\frac{2^{\alpha}\Gamma((d+\alpha)/2)}{\pi^{d/2}|\Gamma(-\alpha/2)|}
|z|^{-d-\alpha}\,, \quad z\in \Rd\,,
\end{equation}
then $L$ is the fractional Laplacian, denoted by $\Delta^{\alpha/2}$ or $-(-\Delta)^{\alpha/2}$.
Back to the case of general $\nu,$
in what follows we write
$$\nu(x,y)=\nu(y-x), \quad x,y\in \Rd.$$
Let $ D\neq \emptyset$ be
an open set in $\Rd$  (further assumptions on $D$ will be added as needed).
Motivated by  Dipierro, Ros-Oton and Valdinoci \cite{MR3651008},
Felsinger, Kassmann and Voigt \cite{MR3318251},
and Millot, Sire and Wang \cite{2016arXiv161007194M}
for $u\colon \Rd \to \R$ we let
\begin{equation}\label{eq:ddf}
\ED(u,u)= \ \ \tfrac12\!\!\!\!\!\!\!\!\!\!\!\!\iint\limits_{\mathbb R^d\times\mathbb R^d\setminus   D^c\times  D^c}
(u(x)-u(y))^2 \nu(x,y)\,\dx\dy.
\end{equation}
Similar expressions are used by Caffarelli, Roquejoffre and Savin \cite[Section~7]{CPA:CPA20331} in the study of minimal surfaces.
For more general L\'evy measures we refer the reader to Rutkowski \cite{MR3738190} and Ros-Oton \cite{MR3447732}.
The quadratic form $\ED(u,u)$ measures the smoothness of $u$ in a Sobolev fashion
by integrating the squared increments of $u$.
The  corresponding Sobolev space is defined as
\begin{equation}\label{eq:dSs}
\vd
=\{u\colon \R^d\to \R \mbox{ such that } \ED (u,u)<\infty\}.
\end{equation}
We also consider
$$\vd_0 = \{u\in \vd: \;u=0 \; a.e.\text{ on } D^c\}.$$
We note that for geometrically regular sets $D$, e.g. Lipschitz open sets, $\vd_0$ can be approximated by the {\it bona fide} test functions, $C_c^\infty (D)$. This is proved in Theorem~\ref{th:density} in the Appendix,
and we refer the reader to
Fukushima, Oshima and Takeda \cite[Section~2.3]{MR2778606} for
a wider context of such approximations.
We consider $\vd$ as the counterpart of the classical Sobolev space $H^1(D)$ of the Laplacian $\Delta$, and $\vd_0$ is a counterpart of $H^1_0(D)$, see, e.g., Evans \cite[Section 5.2]{MR2597943}.\\
We need to note \cite{MR2778606} that
the standard Dirichlet form of $L$ is
\begin{equation}\label{cDf}
\mathcal E_\Rd(u,u)= \ \ \tfrac12\!\!\!\!\iint\limits_{\mathbb R^d\times
	\mathbb R^d} (u(x)-u(y))^2\nu(x,y)\,\dx\dy.
\end{equation}
Clearly,
$\mathcal E_\Rd(u,u)=\ED (u,u)+ \tfrac12\iint_{D^c\times
	D^c} (u(x)-u(y))^2\nu(x,y)\,\dx\dy\ge \ED (u,u)$,
and
$\ED (u,u)=\mathcal E_\Rd(u,u)$ if $u=0$ \textit{a.e.} on $D^c$, in particular if $D=\Rd$.  We let
\begin{equation*}
\ED(u,v) = \ \ \tfrac12\!\!\!\!\!\!\!\!\!\!\!\!\iint\limits_{\mathbb R^d\times\mathbb R^d\setminus   D^c\times  D^c}
(u(x)-u(y))(v(x) - v(y)) \nu(x,y)\,\dx\dy
\end{equation*}
if the integral is absolutely convergent, in particular for  $u,v \in \vd.$ If $u$ and $\phi$ are regular enough, e.g. if $u\in C_c^\infty(\mR^d)$ and $\phi \in C_c^\infty (D),$ then
we have $\ED (u,\phi) = \mathcal{E}_{\mR^d}(u,\phi) = \int_{\mR^d} \phi(x)Lu(x)\,\dx = -\int_{\mR^d} u(x) L\phi(x)\,\dx$  (see the proof of  Lemma~\ref{lem:weakdistr} below).
We emphasize that the increments of $u$ between points $x,y\in D^c$ do not appear in
\eqref{eq:ddf}.
For this reason we consider
$\ED$ and $\vd$  as the
optimal Sobolev setting
for nonlocal
Neumann \cite{MR3651008} and Dirichlet
problems, and we advocate for their use in nonlocal linear and nonlinear PDEs.
Below we prepare ground for such applications, and we focus on the Dirichlet problem:
\begin{equation}\label{eq:Dp}
\left\{
\begin{array}{ll}
Lu=0 & \mbox{ in }  D,\\
u=g & \mbox{ on } D^c.
\end{array}\right.
\end{equation}
We want to solve \eqref{eq:Dp} under minimal smoothness assumptions on the {\it exterior condition} $g\colon D^c\to \R$.
By a solution of
\eqref{eq:Dp} we mean a {\it weak solution}, which is defined as any function $u\in\vd$ equal  to $g$ $a.e.$ on $D^c$, i.e., an {\it extension} of $g$,  such that for all $\phi \in \vd_0$,
\begin{equation}\label{eq:ws-a}
\ED(u,\phi) = 0,
\end{equation}
or, equivalently, $\mathcal{E}_{\mR^d}(u,\phi) = 0$, cf.
\cite{MR3318251,MR3738190}. Even though $u$ does not necessarily belong to $\mathcal V^{\mathbb R^d},$ the integral defining $\mathcal E_{\mathbb R^d}(u,\phi)$ is well defined thanks to the properties of $\phi.$

The interplay of $L$ and \eqref{eq:ws-a} will be visible later on in proofs and results.
In passing we like to mention other viable approaches to
the Dirichlet problem, which do not use the Sobolev setting.
Thus,
we may define
$Lu=0$ as equality of distributions on $D$, to wit,
$\int_{\mR^d} u(x) L\phi(x)\dx=0$ for $\phi\in C^\infty_c(D)$. Alternatively we may replace the condition $Lu=0$, i.e., harmonicity, by a suitable  probabilistic or analytic mean-value property  using the corresponding stochastic process or Poisson and Green kernels.
All the threads will appear in our development below. For additional information we refer
the reader to
Grzywny, Kassmann and Le\.zaj \cite{GKL}.
We note that the special case of $\Delta^{\alpha/2}$ was earlier thoroughly discussed
by
Bogdan and Byczkowski in \cite[Section~3]{MR1671973} and \cite[Lemma~5.3]{MR1825645}; and we refer to Silvestre \cite{MR2270163} for the setting of tempered distributions.

We also note that much care should be exercised when interpreting \eqref{eq:Dp} pointwise or in terms of generators of operator semigroups. Indeed, harmonic functions of reasonable operators may lack sufficient regularity to calculate $Lu(x)$, see Bogdan and Sztonyk \cite[p. 120]{MR2320691}.  On the other hand
even the test functions in $C^\infty_c(D)$, are usually {\it not} in the domain of the generator of the Dirichlet heat kernel  for $D$, see Baeumer, Luks and Meerschaert
\cite[Section~2]{2016arXiv160406421B}.
We further refer to Servadei and Valdinoci \cite{MR3161511}, to Klimsiak and Rozkosz \cite{MR3415028} and to
D\l{}otko, Kania and Sun \cite{MR3338310} for discussions of various notions of solutions to nonlocal equations, and to Barles, Chasseigne, Georgelin and Jakobsen \cite{MR3217703} for several approaches to the nonlocal Neumann problem.

From now on, we assume that $|D^c| \geq 0$. This will be occasionally important, e.g., to select a point in $D^c$ satisfying certain condition which is known to hold $a.e.$
We say that the {\it extension problem}
for $g$, $D$ and $\nu$ (or $L$) has a solution
if the exterior condition $g$ has an extension $u\in\vd$. If this is so, then the existence and uniqueness of the solution to \eqref{eq:Dp} comfortably follow from the general Lax--Milgram theory  (see \cite{MR3318251,MR3738190} and Section \ref{sec:pfs} below).
We may thus focus on the extension problem.
Here the main difficulty is
to define
$u$ on $D$
and
control the
Sobolev smoothness of $u$
 by
that of
$g$. We note here that the extension may as well be used to prove the existence of weak solutions of the non-homogeneous Dirichlet problem
\begin{equation}\label{eq:nhDp}
\left\{
\begin{array}{ll}
Lu=f & \mbox{ in }  D,\\
u=g & \mbox{ on } D^c.
\end{array}\right.
\end{equation}
Here by a weak solution we mean any function $u\in \mathcal{V}^D$ equal to $g$ on $D^c,$ which satisfies $\ED(u,\phi) = \int f\phi\,{\rm d}x$ for every $\phi\in\mathcal{V}_0^D$; this integral  is finite, e.g., if $D$ is bounded and $f\in L^2(D)$, see \cite{MR3318251,MR3738190}.

{Below we characterize
the existence of the solution to the extension problem
by the finiteness of a quadratic Sobolev form $\HD(g,g)$
with a specific
weight $\gamma_D$ on $D^c\times D^c$ defined below and called the {\it interaction kernel}. We also give the corresponding trace theorem. Thus,
we determine functions $g:D^c\to\mathbb R$ that can be extended to a function from $\vd$. In this connection we mention the
result  of
Dyda and Kassmann \cite[Theorem 3]{2016arXiv161201628K}, who
use the Whitney decomposition to solve
the extension and trace problems for $\Delta^{\alpha/2}$. In \cite{2016arXiv161201628K}, the role of  $\gamma_D(x,y)$ is essentially played  by $r(x,y)^{-d-\alpha}$, where
	$r(x,y) = \delta_D(x) + |x-y| + \delta_D(y)$ and $\delta_D(x) = \mbox{dist}\,(x, \partial D)$,  $x, y\in \Rd$. Our result concerns much more general operators $L$. We also prove an identity for the energy of $g$ and the energy of the solution with exterior boundary values $g,$ generalizing the classical {\em Douglas identity} from \cite{MR1501590}.
This result is new even for $\Delta^{\alpha/2}$ in bounded smooth domains.

For information on  extension theorems for local operators we refer the reader to the book of Adams and Fournier \cite{adams2003sobolev}, and
to Ka\l{}amajska and Dhara \cite{MR3458759} and
Koskela, Soto and Wand \cite{2017arXiv170804902K}.

The structure of the paper is as follows.  In
 Section 2 we present the notation, definitions and three main theorems: we start with  Theorem \ref{th:wazne}, which is a Hardy--Stein identity  for $\nu$ (and $L$) generalizing  the formula \cite[(6)]{MR3251822} proved by Bogdan, Dyda and Luks for  $\Delta^{\alpha/2}$,  then we state Theorem \ref{th:toz}, giving the extension and trace  result for $\nu,$ together with the Douglas-type formula
\eqref{eq:SHS}, and finally  we present Theorem~\ref{th:gamma-est}, concerning estimates of the interaction kernel $\gamma_D.$ As a preparation for the proofs,  in Section~\ref{sec:P} we give auxiliary definitions and results.
Then, in Section~\ref{sec:harm} we discuss harmonic functions of $L.$ Several notions of harmonicity are given in Definitions \ref{def:harm} and \ref{def:weakdistr},}  Properties of $L$-harmonic functions are given in Lemma \ref{lem:harmonic}-\ref{th:GKrep}, \ref{lem:harmlu}-\ref{lem:BDL}, and Theorem \ref{lem:harmc2}.  This section is concluded with the proof of  Theorem~\ref{th:wazne}.
In Section~\ref{sec:pfs} we give the proof of the extension and trace  for $\vd$ and
we verify the Douglas formula
from Theorem \ref{th:toz} (see also Corollary \ref{cor:ext}). In Theorem \ref{th:equivalence} we prove the equivalence of various notions of harmonicity for $u\in\mathcal V^D.$
In Section \ref{sec:ikg}  we estimate the interaction kernel $\gamma_D$ for bounded $C^{1,1}$ sets  -- in  Theorem \ref{th:gamma-est}  -- and for  the half-spaces  -- in Theorem \ref{prop:gammaHalfspace}.
In Section \ref{sec:examples} we give specific examples of $\nu$ for which our results apply. In the Appendix we prove auxiliary facts needed to treat $\nu$ and $L$ in the  present generality.
The reader  interested in the general ideas
may focus on $\Delta^{\alpha/2}$. Even in this case the Douglas formula is a
remarkable
conservation law for squared increments of harmonic functions.

In the sequel we will often use the probabilistic language and results
from the potential theory of L\'evy stochastic processes.
This may be avoidable but dramatically reduces the effort needed
 to define and handle such objects as harmonic functions,  Green function and Poisson kernel  for the considered general operators $L$.  Furthermore, the probabilistic setting facilitates integration in spaces with many coordinates and  proofs of the convergence of integral quantities for approximations of $D$ by subsets. Therefore we ask the analytic-oriented reader to bear with us, especially that we managed to largely avoid the language of probability in the statements of our results.

{\bf Acknowledgements.}
We thank  Tomasz D\l{}otko, Bart\l{}omiej Dyda, Damian Fafu\l{}a, Mateusz Kwa\'snicki, Moritz Kassmann, \L{}ukasz Le\.zaj, Tomasz Luks,  Andrzej Rozkosz, and Paul Voigt
for enlightening  discussions. Special thanks are due to Agnieszka Ka\l{}amajska for many discussions on the local Dirichlet problems and extension and trace theorems.

\section{Main results}
Here are
additional assumptions on $\nu\colon [0,\infty)\to (0,\infty]$ which will  sometimes
be made in the sequel.
\begin{description}
\item[A1]
$\nu$ is twice continuously differentiable and there is a  constant $C_1$ such that
\begin{equation*}\label{e:on}
    |\nu'(r)|, |\nu''(r)|\leq C_1\nu(r) \quad \text{for}\; r>1.
    \end{equation*}
\item[A2]  There exist constants $\beta\in (0,2)$ and $C>0$  such that		
\begin{eqnarray}\label{eq:nuSc}\nu(\lambda r)&\leq& C \lambda^{-d-\beta}\nu(r) ,\qquad 0<\lambda,\, r\leq 1,\\ \label{es:nuSc1}
     \nu(r)&\leq& C \nu(r+1), \qquad\quad r\geq 1.
    \end{eqnarray}
\item[{A3}] There exist constants $\alpha\in(0,2)$ and $c>0$  such that
\begin{equation}\label{eq:scalNualoc}\nu(\lambda r)\ \ \geq\ \ c\lambda^{-d-\alpha}\nu(r),
\qquad 0<\lambda,\, r\leq 1.
\end{equation}

\end{description}
Here and below by a {\it constant} we mean a strictly positive number.
Recall that
$\nu(z,w)=\nu(z-w)=\nu(|z-w|)$ for $z, w\in \Rd$. We also denote $\nu(A)=\int_A \nu(|z|)\dz$ and $\nu(z,A) = \nu(A-z)$ for $z\in\mR^d$, $A\subset \Rd$.
We always assume that $L$ and {\it unimodal} $\nu$ are related by \eqref{eq:L-def}.
Clearly, {\rm \textbf{A1}},  {\rm \textbf{A2}} and {\rm \textbf{A3}} hold true if $L=\Delta^{\alpha/2}$. Further examples of L\'evy measure densities $\nu$ satisfying these assumptions are given in Section \ref{sec:examples}.
The condition
{\rm \textbf{A1}} is used for the proof of the
fact that harmonic functions of $L$ (see Definition \ref{def:harm}) are twice continuously differentiable. We note that {\rm \textbf{A1}} implies that for every $s>0$ there
is a (positive finite) constant $C_s$ such that
\begin{equation}\label{e:ons}
|\nu'(r)|, |\nu''(r)|\leq C_s\nu(r), \qquad r\ge s.
\end{equation}
The condition \eqref{eq:nuSc} in {\rm \textbf{A2}} is equivalent to the assumption that $r^{d+\beta}\nu(r)$ is almost increasing on $(0,1]$ in the sense of \cite[Section~3]{MR3165234}, and \eqref{eq:scalNualoc} means that $r^{d+\alpha}\nu(r)$ is almost decreasing on $(0,1]$.

Let $G_D(x,y)$  be the Green function of $D$ for $L$
and
let $\omega_D^x(\cdot)$ be the harmonic measure of $D$ for $L$
(for details see Section~\ref{sec:P}). The first result, a crucial technical tool in our development, is a Hardy--Stein type identity for $L$-harmonic functions, as defined Definition \ref{def:harm} below. The result extends \cite[(6)]{MR3251822}  from $\Delta^{\alpha/2}$ to $L$ and
identifies the
Hardy-type square norm (on the left)
with a Sobolev-type square norm weighted by $G_D$ (on the right).
\begin{theorem}\label{th:wazne}
Assume {\rm \textbf{A1}}. If $u$ is $L$-harmonic in $D$ and $x\in D$, then
		\begin{equation}\label{eq:BDL-1}
		\sup_{x\in U\subset\subset D} \int_{U^c}u^2(z) \, \omega^x_U({\rm d}z) = u(x)^2 + \int_D G_D(x,y)\pord (u(z) - u(y))^2 \nu(z,y)\, \dz \dy.
		\end{equation}
\end{theorem}
Theorem~\ref{th:wazne}
is
proved in Section \ref{sec:harm} by using recent regularity results  of Grzywny and Kwa\'snicki \cite{MR3729529} for $L$-harmonic functions.

We next define
\begin{equation}\label{eq:Pk}
P_D(x,z)=\int_D G_D(x,y)\nu(y,z)\,\dy,\quad x\in D,\ z\in D^c,
\end{equation}
the Poisson kernel of $D$ for $L$. For details on $G_D$ and $P_D$, see Section 3.4.

For $g\colon D^c\mapsto \R$ we
let $P_D[g](x)=g(x)$ for $x\in D^c$ and
\begin{equation}\label{eq:exg}
P_D[g](x)=
\int_{D^c} g(y)P_D(x,y)\,\dy\quad \mbox{ for }\   x\in D,
\end{equation}
if the integrals exists.
 This is the Poisson extension of $g$ and
$\int_{D^c} g(y)P_D(x,y)\dy$
is the Poisson integral.
We define the intensity of interaction of $w,z\in  D^c$ via $D$, in short, the interaction kernel,
$$
\gamma_ D(w,z)
= \int_ D \int_ D \nu(w,x)G_ D(x,y) \nu(y,z)\,\dx\dy= \int_ D \nu(w,x)P_ D(x,z)\,\dx = \int_D \nu(z,x)P_D(x,w)\,\dx.
$$
In particular, $\gamma_D(z,w) = \gamma_D(w,z)$.
The reader may also directly verify the following result.
\begin{example}\label{ex:Cauchy}
{\rm
Let $d=1$, $D=(0,\infty)\subset \R$ and $\nu(w,x)=\pi^{-1}|x-w|^{-2}$, $x,w\in \R$,  i.e.,  $L=\Delta^{1/2}$. Then
$P_{(0,\infty)}(x,z)=\pi^{-1} x^{1/2}|z|^{-1/2}(x-z)^{-1}$ for $x>0$, $z<0$ (see Bogdan \cite[(3.40)]{MR1704245}), and
	  $$\gamma_{(0,\infty)}(z,w)=\int^\infty_0\frac{1}{\pi^2}\frac{\sqrt{x}}
	 {\sqrt{|z|}}\frac{\dx}{(x-z)(x-w)^2}=\frac{1}{2\pi\sqrt{zw}(\sqrt{|z|}+\sqrt{|w|})^2}, \qquad z,w < 0.$$
}
\end{example}

\noindent For
$g\colon D^c\to\mathbb R$
we let
\begin{equation}\label{eq:H-def}
\HD (g,g)=
\tfrac12\!\!\!\!\iint\limits_{ D^c\times  D^c} (g(w)-g(z))^2\gamma_ D (z,w)\,\dw\dz.
\end{equation}
We define
$$\xd=\{g\colon D^c\to\mathbb R:\;
 \mathcal \HD(g,g)<\infty\}.$$
Spaces similar to $\vd$ and $\vd_0$  were considered in \cite{MR3318251, MR3738190}. The space $\xd$ is new ($\mathcal{X}$ stands for  eXterior).
It is an analogue of the classical trace space $H^{1/2}(\partial D)$
\cite{MR1301021}.

Part of our development calls for geometric assumptions on $D$, which we detail in
Section~\ref{sec:P}.
In particular, the $C^{1,1}$ condition for the ``smoothness'' of $D$ and
the volume density condition (VDC) for the ``fatness'' of $D^c$ are defined there.

Our second theorem is in fact the main result of the paper,
on the extension and trace operators between $\vd$ and $\xd$.
\begin{theorem}\label{th:toz}
Let  $D\subset \Rd$ be open, $D^c$ satisfy VDC, $|\partial D|=0$, and
unimodal $\nu$ satisfy {\rm \textbf{A1}}, {\rm \textbf{A2}}.
\begin{itemize}
	\item[{\rm (i)}]
If
$g\in \xd$, then $P_D[g]\in \vd$ and $\ED(P_D[g],P_D[g])=\HD(g,g)$.
\item[{\rm (ii)}] If $u\in \vd,$ then $g = u|_{D^c} \in \xd$ and $\ED(u,u)\ge\HD(g,g)$.
\end{itemize}
\end{theorem}
Thus, under  the assumptions in {\rm (i)}  we have
\begin{equation}\label{eq:SHS}
\tfrac12\!\!\!\!\!\!\!\!\!\! \iint\limits_{\mathbb R^d\times\mathbb R^d\setminus   D^c\times  D^c}
(P_D[g](x)-P_D[g](y))^2 \nu(x,y)\,\dx\dy
=\tfrac12\!\!\!
 \iint\limits_{ D^c\times  D^c} (g(w)-g(z))^2\gamma_D(z,w)\,\dw\dz.
\end{equation}
The proof of Theorem~\ref{th:toz} is given in Section~\ref{sec:pfs} by using Theorem~\ref{th:wazne}.
By analogy with the classical situation \cite[(1.2.18)]{MR2778606}, we call  \eqref{eq:SHS} the {\em Douglas-type formula.}
We note that Douglas integrals for general Dirichlet forms
are also discussed by Fukushima and Chen \cite[Sections~5.5-5.8 and 7.2]{MR2849840} but the form $\mathcal{V}^D$ treated here is new.

\begin{example}\label{ex:Cauchycd}
{\rm
In the setting of Example~\ref{ex:Cauchy}, let $u(x) = g(x)$ for $x\leq 0$, and
$$
u(x)=\int_{-\infty}^0 \frac{\sqrt{x} g(z)\,\dz}{\pi(x-z)\sqrt{|z|}}\quad \mbox{ for } x>0.
$$
Thus, if the above integral is absolutely convergent, then by \eqref{eq:SHS} we get
$$
\iint\limits_{x>0 \text{ or } y>0} \frac{(u(x)-u(y))^2}{\pi(x-y)^2}\,\dx\dy
=\iint\limits_{z<0 \text{ and } w<0}\frac{(g(z)-g(w))^2}{2\pi\sqrt{zw}(\sqrt{|z|}+\sqrt{|w|})^2}\,\dz\dw.$$
}
\end{example}
We note that $\HD(g,g)$ in Theorem~\ref{th:toz} may be finite even for rather rough functions. Indeed,
\begin{align*}
	\HD(g,g) &= \int_{D^c}\int_{D^c}(g(z) - g(w))^2 \gamma_D(z,w)\, \dz \dw \le\ 2\int_{D^c}\int_{D^c} g^2(z) \gamma_D(z,w)\,\dz\dw\\
	&=2\int_{D^c}\int_{D^c} g^2(z) \int_D \nu(z,x)P_D(x,w)\,\dx\dz\dw = 2\int_{D^c} g^2(z) \rho(z)\,\dz,
\end{align*}
where
$\rho(z)=\int_D \nu(z,x)dx$.
In particular, if $g$ is $L^2$-integrable and ${\rm dist}(D,{\rm supp}\, g)>0$,
then $\HD(g,g)<\infty$ and so $g$ has an extension $u\in\vd$. On the other hand $\mathcal{E}_{\mR^d}(u,u) = \infty$ in general for such $g$.
Similarly, if $L=\Delta^{\alpha/2}$ and $D$ is a bounded $C^{1,1}$ set, then $\rho(z)\approx \delta_D(z)^{-\alpha}(1+|z|)^{-d}$, and so
$\HD(g,g)<\infty$ if
$g$ is merely bounded and
$\alpha<1$.

For full analysis of the extension problem  precise estimates of $\gamma_D$ are necessary.
Below we propose sharp explicit estimates of $\gamma_D(z,w)$ for bounded open sets $D$ of class $C^{1,1}$.
For $r>0$ we  let
\begin{eqnarray}\label{eq:Kdef}
K(r)&=&\int_{|z|\leq r}\frac{|z|^2}{r^{2}}\nu(z)\,\dz,\qquad h(r)= K(r)+\nu(B_r^c)
=\int_{\Rd}\left(\frac{|z|^2}{r^{2}}\wedge 1\right)\nu(z)\,\dz,\\
 V(r)&=&\frac{1}{\sqrt{h(r)}}.\label{eq:Vdef}
\end{eqnarray}
Note that $K, h > 0$.
\begin{example}
{\rm
For $\Delta^{\alpha/2}$ we have $K(r)=cr^{\alpha}$ and $V(r)=c'r^{\alpha/2}$ with some constants $c,c'$.
}
\end{example}
Here is our third main result.
\begin{theorem}\label{th:gamma-est}
Let $\nu$ be unimodal and assume {\rm {\rm \textbf{A2}}, {\rm \textbf{A3}}}. Let $D$ be a bounded  $C^{1,1}$ set. Then,
$$\gamma_D(z,w)\approx \begin{cases}
\nu({\delta_D(w)})\hspace{0.03cm}\nu({\delta_D(z)}),&\mbox{ if }\quad \diam(D)\leq {\delta_D(z),\delta_D(w)},\\
\nu(\delta_D(w))/V(\delta_D(z)),&\mbox{ if }\quad {\delta_D(z)}<\diam(D)\leq {\delta_D(w)},\\
\nu(r(z,w))V^2(r(z,w))/\left[ V\left(\delta_D(z)\right)V\left(\delta_D(w)\right)\right], &\mbox{ if }\quad  {\delta_D(z),\delta_D(w)}<\diam(D).
\end{cases}$$
\end{theorem}

As customary in the boundary potential theory, it challenging to handle unbounded and less regular sets $D$, cf. Bogdan, Grzywny, and Ryznar \cite{MR3249349}. In Theorem~\ref{prop:gammaHalfspace} below
we give estimates for $\gamma_H(z,w)$, where $H$ is the half-space in dimensions $d\geq 3$. Other extensions are left for future.

\section{Preliminaries}\label{sec:P}

\subsection{Functions and constants}
In Theorem \ref{th:gamma-est} and below in this paper
we write $f(x)\approx g(x)$, or say that functions $f$ and $g$ are {\it comparable}, if $f, g\ge 0$ and there is a number $C\in (0,\infty)$, called the comparability constant,
such that $C^{-1}f(x)\le g(x)\le C f(x)$ for all the considered arguments  $x$. Such comparisons are also called sharp estimates.
Similarly, $f(x)\lesssim g(x)$  means that $f(x)\leq C g(x)$, the same as  $g(x)\apprge f(x)$.
We write $C=C(a,\ldots,z)$ if the constant $C$
may be so chosen to depend only on $a,\ldots,z$ and we write $C_a$ to emphasize that $C$ may depend on $a$.

We let $C_c(D)$ be the class of continuous functions: $\Rd\to \R$ with compact support contained in $D$  and we let $C_0(D)$  be the closure of $C_c(D)$ in the supremum norm.
By $C_c^\infty (D)$ we denote the class of  infinitely differentiable functions, compactly supported in $D$.
We write $f\in C^2(\overline U)$ if $f\colon\overline U \to \R$ extends to a twice continuously differentiable function in a neighborhood of $\overline{U}$.

\subsection{Geometry}
Recall that $D$ is an open nonempty subset of $\Rd$, the Euclidean space of dimension $d\in \N$, and $D$ has a positive volume.
We write $U\subset\subset D$ if $U$ is an open set, its closure $\overline{U}$ is bounded, hence compact, and $\overline{U}\subset D$.
Let $B(x,r)=\{y\in \Rd: |x-y|<r\}$,
the open ball with radius $r>0$ and center at $x\in \Rd$.
We let $B_r=B(0,r)$ and consider the Poisson kernel of the ball
$P_{B_r}(z):=P_{B_r}(0,z)$, $z\in B_r^c$, see Subsection \ref{sec:S}. We also denote
$B^c_r=(B(0,r))^c$, $\overline{B}^c_r=(\ \! \overline{B(0,r)}\ \! )^c=\{y\in \Rd: |x-y|> r\}$, and $\omega_d=2\pi^{d/2}/\Gamma(d/2)$, the surface measure of the unit sphere in $\Rd$.
\begin{definition}\label{def:VDC}
	We say that $D^c$ satisfies the volume density condition {\rm(VDC)} if there is $c>0$ such that for every $r>0$ and $x\in \partial D$,
	\begin{equation}\label{eq:VDC}
	|D^c \cap B(x,r)| \geq cr^{d}.
	\end{equation}
\end{definition}
This is a fatness-type condition for $D^c,$ uniform in $r$ and $x$. For instance, VDC holds if $D$ satisfies a suitable exterior cone condition.
We say that VDC holds {\it locally} for $D^c$ if VDC holds for $(D\cap B)^c$ for every ball $B$. 	
For instance if $D=\{x\in~\Rd:~ |x|>~ 1\}$, then VDC holds locally for $D^c$.
Naturally, if VDC holds for $D^c$, then VDC holds locally for $D^c$ because $\partial(D\cap B)\subset \partial D\cup \partial B$ and $(D\cap B)^c=D^c\cup B^c$ for every ball $B$. For the sake of the following definition, we denote $\mR^0 = \{0\}$, a singleton.
\begin{definition}\label{def:cont}
	We say that an open set $D\subseteq \mR^d$ has continuous boundary if
		$\partial D$ is compact and
		there exist open sets $U_1,\ldots,U_m,\, \Omega_1,\ldots,\Omega_m\subseteq \mR^d$, continuous functions $f_1,\ldots,f_m\colon \mR^{d-1} \rightarrow \mR,$ and rigid motions $T_1,\ldots,T_m\colon \mR^d \rightarrow\mR^d$, such that
$\partial D \subseteq \bigcup_{i=1}^m U_m,$ and for $i=1,\ldots,m$, we have $T_i(\Omega_i) = \{(x',x_d)\in \mR^{d-1}\times\mR^d: f_i(x') < x_d\}$ and $D\cap U_i =\Omega_i\cap U_i$.
\end{definition}
Thus, locally, $D$ is isometric to $\Omega_i,$ the set above the graph of a continuous function.

A bounded open set is Lipschitz if the functions $f_i$ are Lipschitz: $|f_i(x')-f_i(y')|\le \lambda |x'-y'|$ for $x',y'\in \R^{d-1}$, $i=1,\ldots,m$.
If $D$ is a bounded Lipschitz open set, then VDC holds for $D^c$.
\begin{definition}\label{def:c11-set}
Let $D\subset \Rd$ be open.
If number $R>0$ exists such that for every $Q\in\partial D$ there are balls $B(x',R)\subset D$ and $B(x'',R)\subset D^c$ mutually tangent at $Q$, then $D$ is $C^{1,1}$ (at scale $R$).
\end{definition}
If $D$ is of class $C^{1,1},$ then it is Lipschitz and the defining functions $f_i$ can be so chosen that their gradient is Lipschitz, see Bogdan and Jakubowski \cite{MR2892584} for more on the geometry of $C^{1,1}$ open sets.

\subsection{Completeness}

\begin{lemma}\label{lem:l2d}  We have $\vd\subseteq L^2_{loc}(\mR^d)$.
	If $D$ is bounded, then $\vd\subseteq L^2(D)$.
\end{lemma}
\begin{proof}
Let $\emptyset \neq U\subseteq D$ be  open and bounded. For $u\in \vd$ we have
	\[\int_D\int_{\mathbb R^d} (u(x)-u(y))^2 \nu(x,y)\,\dx\dy<\infty.\]
	In particular there is a $y_0\in D$ such that
	\[\int_{ U} (u(x)-u(y_0))^2\nu(x,y_0)\,\dx<\infty,\]

Since $\nu$ is unimodal and strictly positive and $U$ is bounded,
we have $\nu(x,y_0)\geq c>0$ for $x\in U$. Consequently, $\int_U (u(x)-u(y_0))^2\,\dx<\infty$.
For every $a,b\in \mathbb R$ we have $a^2\leq 2(a-b)^2+2b^2,$ hence
\[\int_U u(x)^2\,\dx\leq 2\int_U (u(x)-u(y_0))^2\,\dx+2|U|u(y_0)^2<\infty.\]

For bounded $D,$ the argument above holds true with $U$ replaced by $D.$
\end{proof}

In view of Lemma~\ref{lem:l2d} for bounded $D$  it is plausible to let
\begin{equation}\label{eq:norm2}
\|u\|_{\vd}=\sqrt{\|u\|_{L^2(D)}^2+ \ED (u,u)}.
\end{equation}
This is a seminorm, actually a norm on $\vd$, because if a nonzero function $u$ vanishes {\em a.e.} in $D$, then by the strict positivity of $\nu$, the increments between $D^c$ and $D$ yield a positive value of $\ED(u,u)$.
Furthermore, $\vd_0$
is a Hilbert space with this norm \cite[Lemma~2.3]{MR3318251}, \cite[Lemma~3.4]{MR3738190}. The completeness of $\vd_0$ is also a consequence of the completeness of $\vd$. The latter is not given in \cite{MR3318251, MR3738190}, but it was verified in \cite{MR3651008} for the fractional Laplacian. We present a short proof which  uses only the fact that $\nu$ is locally bounded away from zero.
\begin{lemma}\label{lem:vhil}
	If $D$ is bounded, then $\vd$ is complete with the norm $\|\cdot\|_{\vd}$.
\end{lemma}
\begin{proof}
If $\emptyset \ne U\subset\subset D$, then
\begin{align*}
 &\int_{D^c} u(y)^2 \nu(y,U) \,\dy = \int_{U}\int_{D^c} u(y)^2\nu(x,y)\,\dy\dx \\\leq\ &2 \int_{U}\int_{D^c}(u(x) - u(y))^2\nu(x,y)\,\dy\dx
  + 2 \int_{U}\int_{D^c}u(x)^2\nu(x,y)\,\dy\dx\\
\leq\ &4 \ED(u,u) + 2\int_{U}u(x)^2\nu(x,D^c)\,\dx \lesssim \|u\|_{\vd}^2.
\end{align*}
The last inequality follows from the fact that $x \mapsto \nu(x,D^c)$ is bounded on $U$. 
Thus the norm $\|\cdot \|_{\mathcal{V}^D}$ dominates the norm in $L^2((\textbf{1}_D(y) + \nu(y,U)\textbf{1}_{D^c}(y))\, \dy)$. Furthermore, $y \mapsto \nu(y,U)$ is locally bounded from below (by a positive constant) on $D^c$. Therefore every Cauchy sequence in $\mathcal{V}^D$ has a subsequence that converges to some measurable $u$ \textit{a.e.} in $\mR^d$. By Fatou's lemma, $\|u\|_{\vd} < \infty$ and also $\|u_n - u\|_{\vd} \to 0$ as $n\to\infty$, cf. \cite[Lemma 2.3]{MR3318251}.
\end{proof}
\begin{lemma}\label{lem:l1}
	If \eqref{es:nuSc1} holds, then $\mathcal{V}^D\subset L^2(1\wedge \nu)\subset  L^1(1\wedge \nu)$.
\end{lemma}
\begin{proof}
For $D_1 \subseteq D_2 \subseteq \mR^d$ we have $\mathcal{V}^{D_2} \subseteq \mathcal{V}^{D_1}$, so we may assume that $D$ is bounded. Fix nonempty open $U\subset\subset D$ and $x_0\in U$. By \eqref{es:nuSc1} we have $\nu(y,U) \approx \nu(y,x_0)$ for $y\in D^c$. The result follows from Lemma  \ref{lem:l2d}, the proof
of   Lemma~\ref{lem:vhil} and the finiteness of the measure $1\wedge\nu(x)\,\dx$.
\end{proof}

\subsection{Stochastic process}\label{sec:S}
We define
$$
\psi(\xi)=\int_\Rd (1-\cos \xi\cdot x)\,\nu(|x|)\,\dx, \quad \xi \in \Rd,
$$
the L\'evy--Khinchine exponent for $\nu$.
Since $\nu(\mR^d) = \infty$, by K.-i. Sato \cite[Theorem 27.7]{MR1739520} and Kulczycki, Ryznar \cite[Lemma 2.5]{MR3413864}, for every $t>0$ there is a continuous function $p_t(x)\geq 0$   on $\Rd\setminus\{0\}$ such that
$$
\int_\Rd \me^{i\xi\cdot x} p_t(x)\,\dx=\me^{-t\psi(\xi)}, \quad \,\xi\in \Rd.
$$
Measures $\mu_t({\rm d}x)=p_t(x){\rm d}x$ form a weakly continuous convolution semigroup on $\mathbb R^d.$
Accordingly,
$$P_t f(x)=\int_\Rd f(y)p_t(y-x)\,\dy,\quad t>0,\ x\in \Rd,$$
is a strongly continuous semigroup of operators on
$C_0(\Rd)$ and its generator is a Fourier multiplier with the symbol $-\psi(\xi)$, cf. \cite[Chapter 6]{MR3587832}.
On Borel sets in the space $\Omega$ of c\`adl\`ag functions $\omega\colon [0,\infty)\mapsto \Rd$ we consider the probability measures $\mpr^x$, $x\in \Rd$,  constructed by the Kolmogorov's extension theorem from the finite-dimensional distributions
$$
P^x_{t_1,t_2,\ldots, t_n}(A_1,\ldots,A_n)=\int_{A_1}\cdots \int_{A_n} \prod_{i=1}^n p_{t_i-t_{i-1}}(x_i-x_{i-1}) \, \mathrm{d}x_n \cdots \mathrm{d}x_1,
$$
where $0=t_0\le t_1\le \ldots \le t_n$, $x_0=x$, $A_1,\ldots A_n \subset \Rd$ and $n=1,2,\ldots$ \cite[p. 54]{MR1739520}.
The {\it process} $X_t(\omega):=\omega(t)$ on $\Omega$ is a convenient tool to handle $\mpr^x$.
In particular,  $\mpr^x (X_t\in A)=\int_ A p_t(y-x)\,\dy$ \linebreak and
$\mpr^x(X_{t_1}\in A_1,\ldots,X_{t_n}\in A_n)=P^x_{t_1,t_2,\ldots, t_n}(A_1,\ldots,A_n)$.
We call $\mpr^x$ the distribution of the process starting from $x\in \Rd$ and we let $\mE^x$ be the corresponding integration.
By the construction, $X= \{X_t\}_{t\ge 0}$ is a symmetric L\'evy process in $\Rd$
with $(0,\nu,0)$ as the L\'evy triplet \cite[Section 11]{MR1739520}.
We let, as usual, $X_{t-} = \lim_{s\to t^-} X_s$ for $t>0$ and $X_{0^-}=X_0$.
We introduce the time of the first exit of $X$ from $D$,
$$\tau_D=\tau_D(X)=\inf\{t\ge 0: \, X_t\notin D\}.$$
We also consider the usual shift operators $F(X)\circ \theta_s=F(\{X_{t+s}, t\ge 0\})$, $s\ge 0$, cf. Blumenthal and Getoor \cite{MR0264757} or Chung, Zhao \cite{MR1329992}.
The Dirichlet heat kernel $p_t^D(x,y),$
is determined by the identity
$$
\int_\Rd f(y)p_t^D(x,y)\,\dy=\mE^x [f(X_t);\tau_D>t],\quad t>0,\ x\in\mR^d,
$$
where $f\colon\Rd\to [0,\infty]$, cf.  Bogdan, Grzywny, Ryznar \cite[Section 1.3]{MR3249349} and \cite{MR1329992}.
The Green function of $D$ is
$$
G_D(x,y)=\int_0^\infty p_t^D(x,y)\,\dt,\quad x,y\in\mR^d,
$$
and for functions $f\ge 0$ we have
$$
\int_\Rd G_D(x,y)f(y)\,\dy=\int_0^\infty \int_\Rd f(y)p_t^D(x,y)\,\dy=\mE^x \int_0^{\tau_D} f(X_t)\,\dt,\quad x\in\mR^d.
$$
Accordingly, $G_D(x,y)$ is interpreted as the occupation time density of $X_t$ prior to the first exit from $D$.
The following Ikeda--Watanabe formula defines the joint distribution of $(\tau_D,X_{\tau_D-},X_{\tau_D})$ restricted to the event $\{\tau_D<\infty,X_{\tau_D-}\neq X_{\tau_D}\}$: if
$x\in D$, then
\begin{align}\label{eq:IW}
\mpr^x[\tau_D\in I,\; A\ni X_{\tau_D-}\neq X_{\tau_D}\in B]=
\int\limits_I \int\limits_{B} \int\limits_A p_u^D(x,y)\nu(y,z)\,\dz\dy\mathrm{d}u,
\end{align}
see, e.g., Bogdan, Rosi\'nski, Serafin, Wojciechowski \cite[Section~4.2]{MR3737628}.
Thus, if ${\rm dist} (B,D)>0$, then
\begin{align}\label{eq:IWG}
\mpr^x[X_{\tau_D}\in B]=
\int_B P_D(x,z)\,\dz \leq 1,\quad x\in D,
\end{align}
where $P_D$ is the Poisson kernel \eqref{eq:Pk}.
The $L$-harmonic measure of $D$ for $x\in \Rd$, denoted $\omega_D^x,$
is the distribution of the random variable $X_{\tau_D}$ under $\mathbb P^x$.
Thus,
\begin{equation}\label{eq:hm}
\omega_D^x({\rm d}z) =\mpr^x[X_{\tau_D}\in {\rm d}z].
\end{equation}
From \eqref{eq:IW} we see that $P_D(x,z){\rm d}z$ is the part of $\omega_D^x({\rm d}z)$  which results from the discontinuous exit (by a jump) from $D$. In particular, for sufficiently regular $D$, by Lemma \ref{l:nuSc} below we have
\begin{equation}\label{eq:ikedawatanabe}
\omega_D^x(\dz) = P_D(x,z)\, \dz = \int_D G_D(x,y)\nu(y,z)\, \dy \dz.
\end{equation}
The reader may easily obtain other marginal distributions of $(\tau_D,X_{\tau_D-}, X_{\tau_D})$. For instance,
\begin{equation}\label{eq:skad1}
 \mpr^x(X_{{\tau_ D}^-}\in  D)=\int_ D  G_ D(x,y)\kappa_ D(y)\,\dy,\quad x\in  D,
\end{equation}
where
$$\kappa_ D(x)=\int_{ D^c} \nu(x,z)\,\dz, \quad x\in D.$$
Formula \eqref{eq:IW} allows to interpret $p^D_u(x,y)$ as the density function of the distribution
of
$X_u$ for the process killed (disappearing) at time $\tau_D$.
We interpret
$\kappa_D(x)$
as the intensity of killing (escape outside $D$).
For $U\subset D$ we have inequalities $p^U\le p^D$, $G_U\le G_D$. Also,
$P_U(x,z)\le P_D(x,z)$ for $x\in U$, $z\in D^c$, and $\gamma_U(z,w)\leq \gamma_D(z,w)$ for $z,w\in D^c$.
These inequalities are referred to as \textit{domain monotonicity}.
\section{Harmonic functions}\label{sec:harm}
Let $L$ be the operator given by \eqref{eq:L-def} and let $(X_t, \mathbb P^x)_{t\geq 0,\, x\in\mathbb R^d}$ be the symmetric pure-jump L\'{e}vy process in $\Rd$ constructed above.
As before, $D$ denotes a fixed nonempty open subset of $\Rd$.

\begin{definition}\label{def:harm} {\rm (i)}
We say that $u\colon \mathbb R^d\to\mathbb R$ is $L$-harmonic (or harmonic, if $L$ is understood) in $D$ if it has the mean value property, that is for all (open) $U\subset\subset D$
and $x\in U$,
$$u(x) = \mathbb{E}^x u(X_{\tau_U}).$$
{\rm (ii)} We say that $u$ is regular $L$-harmonic (or regular harmonic) in $D$ if
 $u(x) = \mathbb{E}^x u(X_{\tau_D})$ for $x\in D$.\\
In {\rm (i)} and {\rm (ii)} we assume that the integrals are absolutely convergent.

\end{definition}

\begin{lemma}\label{lem:harmonic}
If $u$ is regular $L$-harmonic in $D$, then it is  $L$-harmonic in $D$.
\end{lemma}
\begin{proof}
Let $g\colon D^c\to [0,\infty]$.
Let $u(x)=g(x)$ for $x\in D^c$ and $u(x)=\mathbb E^x g(X_{\tau_D})$ for $x\in D$.
Thus, $u$ is regular harmonic on $D$.
Let $U$ be an arbitrary open set such that $U\subset D$. Of course, $\tau_U\leq \tau_D$. 
Note that $\tau_D=\tau_U+\tau_D\circ \theta_{\tau_U}$ and $u(X_{\tau_D})=u(X_{\tau_D})\circ \theta{\tau_U}$. Consider the usual $\sigma$-algebra ${\mathcal F}_{\tau_U}$ of up-to-$\tau_U$ events   \cite{MR0264757, MR1329992}.
Let $x\in U$. 
By the tower rule of conditional expectations, the strong Markov property of $X$
and the regular harmonicity of $u$, 
\begin{eqnarray*}
u(x)&=& \mathbb E^x u(X_{\tau_D}) =
\mathbb E^x [ \mathbb E^x [ u(X_{\tau_D})\circ \theta_{\tau_U}|\mathcal F_{\tau_U}]]\\
&=&\mathbb E^x [ \mathbb E^{X_{\tau_U}} [u(X_{\tau_D})]]
=\mathbb E^x[u(X_{\tau_U})].
\end{eqnarray*}
In particular $u$ is harmonic on $D$.
The case of general (signed) $u$ follows from the above by taking $g$ equal to
the positive and negative parts of $u$ on $D^c$.
\end{proof}
\begin{lemma}\label{lem:preharmonic}
If $u(x)=\mathbb E^x [u(X_{\tau_D}); X_{\tau_{D-}}\neq X_{\tau_D}]$ for all $x\in D$, then $u$ is  harmonic in $D$.
\end{lemma}
The proof is similar to that of Lemma~\ref{lem:harmonic}, so we skip it.
\begin{remark}\label{rem:martingale}
{\rm The proof of Lemma~\ref{lem:harmonic} in fact shows that
$\{u(X_{\tau_U}), U\subset D\}$ is a martingale ordered by inclusion of open subsets of $D$; the martingale is closed by $u(X_{\tau_D})$ if $u$ is regular harmonic. 
}
\end{remark}

\begin{lemma}\label{lem:harml1}
	If $u$ is
	$L$-harmonic in $D,$ then $u\in L^1_{loc} (\mR^d)$.
\end{lemma}
	\begin{proof} 
Let $0<\epsilon < \mbox{dist}\,(x,D^c)$. Then,
	$\int_{B_{\eps}^c(x)} |u(z)| P_{B_{\eps}(x)}(x,z)\, \dz<\infty$. 
By Ikeda--Watanabe, $P_{B_{\eps}(0)}(0,z)>0$ on $B^c_{\eps}(0)$.
By \cite[Corollary 2.4]{MR3729529}, $z\mapsto P_{B_{\eps}(0)}(0,z)$ is radially nonincreasing on $B^c_{\eps}(0)$, so
$P_{B_{\eps}}(x,z)$ is locally bounded away from zero on $B_{\eps}^c(x)$.
The result easily follows.
	\end{proof}

Here is more on the $L^2$-integrability
implied by the $L^2$-integrability of increments, cf. Lemma~\ref{lem:l2d}.

\begin{lemma}\label{lem:integrability}
Assume that {\rm \textbf{A2}} holds. If $g\in\xd$ and $x\in D$, then $\int_{D^c} g(z)^2P_D(x,z)\,\dz<\infty$.
\end{lemma}
\begin{proof} By the definition of $\gamma_D$,
\begin{eqnarray}
\HD(g,g) & =& \tfrac{1}{2}\int_{D^c}\int_{D^c}\int_D (g(z)-g(w))^2 \nu(w,x)P_D(x,z)\,\dx\dz\dw<\infty.\label{eq:*}
\end{eqnarray}
Since $\nu>0$,
for almost all $(x,w)\in D\times D^c$ we obtain
\begin{eqnarray}\label{eq:finiteness}
\int_{D^c} g(z)^2 P_D(x,z)\,\dz&\leq&2\int_{D^c} (g(w)-g(z))^2 P_D(x,z)\,\dz+2g(w)^2<\infty.
\end{eqnarray}
Thus $\int_{D^c} g(z)^2P_D(x,z)\,\dz<\infty$ for almost every $x\in D$. {\rm \textbf{A2}} lets us use the boundary Harnack principle given by Grzywny and Kwa\'snicki in \cite[(1.12)]{MR3729529} to get this assertion for all $x\in D.$ 
Indeed, let $n=1,2,\ldots$, $u_n(x)=g_n(x)=g^2(x)\wedge n$ for $x\in D^c$ and 
$u_n(x)=\mathbb E^x[g_n(X_{\tau_D}); X_{\tau_D-}\neq X_{\tau_D}]$ otherwise.
Similarly, we let $u(x)=g^2(x)$ if $x\in D^c$, elsewhere we let 
$u(x)=\mathbb E^x[g(X_{\tau_D})^2; X_{\tau_D-}\neq X_{\tau_D}]$.
Clearly, $u=\lim u_n$. These functions are (finite and) regular harmonic on every  $U\subset \subset D$. By
Lemma~\ref{lem:preharmonic} and \cite[(1.12)]{MR3729529} the functions $u_n$ are uniformly in $n$ locally bounded on $D$, because $u_n\le u$. It follows that $u$ is locally bounded on $D$, in particular it is finite on $D$. 
\end{proof}

We fix an arbitrary (reference) point $x_0\in D.$ For $g\in\xd,$ we let $$|g|_{D^c}^2=\int_{D^c} g(z)^2P_D(x_0,z)\,\dz$$ (we  omit $x_0$ from the notation).
The expression is finite by Lemma \ref{lem:integrability}.
We define a norm on $\xd$:
\begin{equation}\label{eq:norm1}
\|g\|_{\xd}=\sqrt{|g|_{D^c}^2+ \HD(g,g)}.
\end{equation}
Arguing as in the last part of the proof of Lemma \ref{lem:vhil} we see that $\xd$ is complete with this norm.

\noindent The next result is due to Grzywny and Kwa\'snicki \cite{MR3729529}.
\begin{lemma}\label{th:GKrep}
Let $0\leq q < r<\infty$. There is a radial kernel $\olp_{q,r}(z)$, a constant $C = C(d,\nu,q,r) > 0$ and a probability measure $\mu_{q,r}$ on the interval $[q,r]$, such that
\begin{equation}\label{eq:Pqr}
	\olp_{q,r}(z) = \int_{[q,r]} P_{B_s}(z)\,\mu_{q,r}(\ds)=\int_{[q,r]}\int_{B_s}\nu(y,z)G_{B_s}(0,y)\,\dy\,\mu_{q,r}(\ds),\quad |z| > r,
	\end{equation}
$\olp_{q,r} = 0$ in $B_q$, $0\leq \olp_{q,r}\leq C$ in $\mR^d$, $\olp_{q,r} = C$ in $B_r\setminus B_q$ and $\olp_{q,r}$ decreases radially on $B_r^c$. Furthermore, $\olp_{q,r}(z) \leq P_{B_r}(z)$, for $|z| > r$, and if $f$ is $L$-harmonic in $B_r$, then
$$f(0) = \int_{\mR^d\setminus B_q} f(z)\olp_{q,r}(z)\, \dz.$$
\end{lemma}
\begin{corollary}\label{rem:1}
If $f$ is
$L$-harmonic in $B_{2r},$ then  $f = f\ast \olp_{0,r}$ in $B_r.$
\end{corollary}
We will use Lemma \ref{th:GKrep} to prove that Poisson extensions are twice continuously differentiable under the additional assumption {\rm \textbf{A1}}. In the proof we closely follow the arguments from Theorem 1.7 and Remark 1.8 b) in \cite{MR3729529} except that we do not assume the boundedness of $u$.

\begin{theorem}\label{lem:harmc2} Suppose that $\nu$ satisfies {\rm \textbf{A1}}
 and let $
D\subset \mathbb R^d$ be an open set.
	If $u\colon \mathbb R^d\to\mathbb R$  is $L$-harmonic in $D,$
	e.g., if $u(x) = \podc u(z) P_D(x,z)\,\dz$ for $x\in D$, then $u\in C^2(D)$.
\end{theorem}
	\begin{proof}
Note that {\rm \textbf{A1}} yields \eqref{es:nuSc1}. We are in a position to apply Lemma \ref{th:GKrep}. Let $x\in D,$
and let $r>0$ be such that $B_{2r}(x)\subset D.$  Since $\nu(z)$ is continuous,  we get from \eqref{eq:Pqr} that kernels $\overline{P}_{q,r}$ are continuous. By Corollary~\ref{rem:1}, $u$ is continuous in $B_r(x).$
	Next we fix a nonnegative smooth radial function $\kappa$ such that $0\leq\kappa\leq 1$, $\kappa \equiv 1$ in $B_{\frac 32 r}$ and $\kappa \equiv 0$ outside $B_{2r}$. As in \cite{MR3729529}, we denote $\pi_r(z)=\overline{P}_{0,r}(z)\kappa(z)$ and $\Pi_r(z)=\overline{P}_{0,r}(z)(1-\kappa(z)).$  Obviously, $u = \Pi_r \ast u + \pi_r\ast u$ in $B_r(x).$ In particular, both terms are well-defined. Iterating, we get
	\begin{align}
 u&= \big(\Pi_r + \pi_r\ast\Pi_r + \pi_r^{\ast 2}\ast\Pi_r + \ldots \pi_r^{\ast (k-1)}\ast\Pi_r + \pi_r^{\ast k}\big)\ast u\nonumber\\
&= (\delta_0 + \pi_r + \pi_r^{\ast 2} + \ldots + \pi_r^{\ast (k-1)})\ast\Pi_r \ast u + \pi_r^{\ast k}\ast u.\label{eq:poissexp}
	\end{align}
Using an argument based on the Fourier transform as in \cite[Proof of Theorem 1.7]{MR3729529}, we get that for every $N$ there is a sufficiently large $k$, such that the function $\pi_r^{\ast k}$ is $N$ times continuously differentiable. It is also compactly supported. Since $u\in L^1_{loc}(\mR^d)$, it follows that $ \pi_r^{\ast k}\ast u$ is $N$ times continuously differentiable in $D$. For our purposes below, it suffices to take $N=2$.

We will now handle the first summand in \eqref{eq:poissexp}. First, observe that if $\theta>r,$  $|z|>\theta >r,$ and $|\alpha|\in \{1, 2\}$, then
\begin{equation}\label{eq:por}
\left|\partial^\alpha \overline{P}_{0,r}(z)\right|\leq C_{\theta,r} \overline{P}_{0,r}(z).
\end{equation}
Indeed, by the definition of $\olp_{0,r}$ and the Ikeda--Watanabe formula we have
\begin{equation*}
\olp_{0,r}(z) = \int_{[0,r]}P_{B_s}(z)\, \mu_{0,r}(\ds) = \int_{[0,r]}\int_{B_s} \nu(y,z)G_{B_s}(0,y)\,\dy\,\mu_{0,r}(\ds)	
\end{equation*}
and further
\[\partial^\alpha \overline{P}_{0,r}(z)=\int_{[0,r]}\int_{B_s}
\partial^\alpha_z\nu(y,z)G_{B_s}(0,y)\,\dy\,\mu_{0,r}(\ds).\]
For  $z$ as above and $y\in B_s\subset B_r$ we have  $|z-y|\geq \theta-r$.
By {\rm \textbf{A1}}
\begin{eqnarray*}
|\partial^\alpha \overline{P}_{0,r}(z)|&\leq & C_{\theta,r} \int_{[0,r]}\int_{B_s}
\nu(y,z)G_{B_s}(0,y)\,\dy\,\mu_{0,r}(\ds)= C_{\theta,r}\overline {P}_{0,r}(z).
\end{eqnarray*}
Since $\mbox{supp}\, \Pi_r\,\subset B^c_{\frac{3}{2}r},$ and $\kappa$ is smooth, from the Leibniz rule and \eqref{eq:por} we see that for all $z\in\mathbb R^d,$
$
|\partial^\alpha \Pi_r(z)|\leq C_r|\Pi_r(z)|.
$ Therefore if $|\alpha|\le 2$, then
\[\int_{\mathbb R^d}|\partial^\alpha\Pi_r(x-z)u(z)|\,\dz<\infty,\]
which allows to differentiate under the integral sign and so
$\partial^\alpha \Pi_r\ast u(x)$ is well-defined. Continuity of the derivative follows from the continuity of $\partial^\alpha\nu$ and the dominated convergence.
\end{proof}
\begin{lemma}\label{lem:harmlu}
	Assume {\rm \textbf{A1}}. If $u$ is $L$-harmonic in $D$, then $Lu=0$ on $D$.
\end{lemma}
\begin{proof}
	By Theorem~\ref{lem:harmc2}, $u\in C^2(D)$.
	Let $x\in U\subset \subset D$.
	Let $\phi\in C_c^2(D)$ be such that $u=\phi$ on $U$. Let $w=u-\phi$.
	We recall that on $C^2_c(\Rd)$, $L$ coincides with the generator of the semigroup $\{P_t\}$   \cite[Theorem~31.5]{MR1739520} and also
	with the Dynkin  characteristic operator \cite[Chapter~V.3]{MR0193671}
	\[\mathcal U \phi(x)=\lim_{r\to 0}\frac{\mathbb E^x \phi(X_{\tau_{B(x,r)}})-\phi(x)}{\mathbb E^x\tau_{B(x,r)}}.\]
	Since $w=0$ in a neighborhood of $x$, by Corollary~\ref{cor:ac} we get
	\begin{align*}
	\mathcal U w(x)&=\lim_{r\to 0}\frac{\mathbb E^x w(X_{\tau_{B(x,r)}})}{\mathbb E^x\tau_{B(x,r)}}\\
	&=\lim_{r\to 0}\frac1{\mathbb E^x\tau_{B(x,r)}}\int_{B(x,r)^c}\int_{B(x,r)} G_{B(x,r)}(x,z)\nu(z,y)\,\dz w(y)\,\dy\\
	&=\lim_{r\to 0}\frac1{\mathbb E^x\tau_{B(x,r)}}\int_{B(x,r)} G_{B(x,r)}(x,z)\int_{U^c}\nu(z,y)w(y)\,\dy\dz.
	\end{align*}
	By Lemma \ref{lem:harml1} we have $\int_{B(x,r)^c}\nu(x,y)|u(y)|\,\dy<\infty$ for $r>0$.
	It follows that $z\mapsto \int_{U^c}\nu(z,y)w(y)\,\dy$ is a bounded continuous function near $x$.
	Since
	$\mathbb E^x \tau_{B(x,r)}=\int_{B(x,r)} G_{B(x,r)}(x,z)\,\dz$, we see that $G_{B(x,r)}(x,z)\,\dz/\int_{B(x,r)} G_{B(x,r)}(x,z)\,\dz$ converges weakly to the Dirac mass at $x$ as $r\to 0$. Therefore,
	$\mathcal U w(x)=\int_{U^c}\nu(x,y)w(y)\,\dy = \int_{\mR^d} (w(y) - w(x))\nu(x,y) \,\dy = Lw(x)$.
	We get
	$$Lu(x)=L\phi(x)+Lw(x)=\mathcal U \phi(x)+\mathcal U w(x)=\mathcal U u(x).$$
	On the other hand, by the mean value property of $u$ we have $\mathcal U u(x)=0$. Therefore $L u(x)=0$.
\end{proof}
We should warn the reader that
for operators $L$ more general than those considered here, $L$-harmonic functions may lack sufficient regularity to calculate $L u$ pointwise, see remarks after Corollary~20 in Bogdan, Sztonyk \cite{MR2320691}.

\begin{lemma}\label{lem:dynk}
		Assume that \eqref{es:nuSc1} holds. Let $U\subset\subset D$ be open and Lipschitz. Let $u \in C^2(\oU)$ and $\int_{\mR^d} |u(y)| (1\wedge\nu(y))\, \dy < \infty$.
  Then $Lu$  is bounded on $\oU$ and
  for every $x\in \mR^d$,
		\begin{equation}\label{eq:A1}
		\mE^xu(X_{\tau_U}) - u(x) =  \int_U G_U(x,y) L u(y)\, \dy.
		\end{equation}
	\end{lemma}
\begin{proof}
 Both sides of \eqref{eq:A1} are equal to zero for $x\notin U$, so let $x\in \oU.$
To prove that $Lu(x)$ is bounded on $\overline{U}$ we choose $\eps>0$ so small that $u$ is $C^2$ on $\oU + B_{2\eps}.$ In particular $u$ and its second-order partial derivatives $D^2 u$ are bounded on $\oU + B_\eps$.
As usual by Taylor's formula,
		\begin{eqnarray*}
		|Lu(x)| &=& \left|\tfrac12\pord (2u(x) - u(x+y) - u(x-y) )\nu(y)\, \dy\right|\\
		&\leq& \tfrac12\sup_{\stackrel{\xi\in \overline{U}+B_\epsilon}{|\alpha| = 2}}|\partial^\alpha u(\xi)|\int_{B_\epsilon} |y|^2\nu(y)\,\dy + \tfrac12\int_{B_\epsilon^c}\left| (2u(x) - u(x+y) - u(x-y) \right|\nu(y) \,\dy
		\\
&\leq& C_\epsilon \int_{B_\epsilon} |y|^2\nu(y)\,\dy
+ |u(x)|\nu(B^c_\epsilon)+ \int_{B^c_\epsilon} |u(x+y)|\nu(y)\,\dy.
		\end{eqnarray*}
We only need to estimate the last integral. Let $R= \eps + \sup\limits_{x\in U}|x|.$
Then,
\begin{eqnarray}\label{eq:bound-L}
\int_{B_\epsilon^c}|u(x+y)|\nu(y)\,\dy &=& \int_{B_\epsilon^c(x)}|u(z)|\nu(x,z)\,\dz\nonumber\\
&=&\int_{B_\epsilon^c(x)\cap B_{2R}}|u(z)|\nu(x,z)\,\dz+ \int_{B_\epsilon^c(x)\cap B^c_{2R}}|u(z)|\nu(x,z)\,\dz.
\end{eqnarray}
 The first integral in \eqref{eq:bound-L}
 is not greater than $\nu(\epsilon)\int_{B_{2R}}|u(z)|\,\dz<\infty$.
For the second integral
we note that  $x\in\overline U,$ $z\notin B_{2R},$ imply $|z-x|\geq |z|-|x|\geq |z|-R$. From \eqref{es:nuSc1} there is $C_R>0$ such that
$\nu(z,x)\leq C_R\nu(z)$ and so the integral is bounded by $C_R\int_{B_{2R}^c}|u(z)|\nu(z)\,\dz<\infty$.

Collecting all the bounds together we see that
\begin{eqnarray*}
|Lu(x)|&\leq& C_\eps\int (|y^2|\wedge 1)\nu(y)\,\dy+2\|u\|_{\infty}\nu(B_\epsilon^c)\\
&&+ 2C_\epsilon\int_{B_{2R}}|u(z)|\,\dz + 2M\int_{B_{2R}^c}|u(z)|\nu(z)\,\dz.
\end{eqnarray*}

For the second part of the statement, let $u = \varphi + h$, where $\varphi \in C_c^2(\mR^d)$ and $h = 0$ in a neighborhood of $U$. Note that $h\in C^2(\oU)$. We claim that \eqref{eq:A1} holds with $\varphi$ and $h$ in place of $u$. Recall that $L$ coincides with the generator of the semigroup of $X_t$ for functions from $C_c^2(\mR^d)$. Therefore, by Dynkin's formula \cite[(5.8)]{MR0193671},
 		\begin{equation}\label{eq:dyn}
		\mE^x \varphi(X_{\tau_U}) - \varphi(x) = \mE^x \int_0^{\tau_U} L\varphi(X_t) \,\dt= \int_0^\infty \mE^x \left[L\varphi(X_t); \tau_U>t\right]\,\dt.
		\end{equation}
Here the change of the order of integration is justified because $L\varphi$ is bounded on $\overline U$  and $\mE^x \tau_U < \infty$,
cf., e.g., \cite{MR3350043,pruitt1981}.
As usual, we let $p^U$ denote the transition density of the process killed upon leaving $U$. Since $L\varphi$ is measurable and bounded on $\overline{U},$ for every $t>0$ we have
\[\mE^x \left[L\varphi(X_t); \tau_U>t\right] = \int_U p^U_t(x,y) L\varphi (y)\,\dy.\]
Therefore
\[\mE^x\varphi(X_{\tau_U})-\varphi(x)= \int_0^\infty\int_Up^U_t(x,y)L\varphi(y)\,\dy{\rm d}t = \int_U G_U(x,y)L\varphi(y)\,\dy,\]
which proves the claim for $\varphi$.

Let $x\in U$. By the Ikeda--Watanabe formula we have
\begin{align*}
\mE^x h(X_{\tau_U}) - h(x) = &\int_{U^c} h(z) P_U(x,z)\, \dz = \int_{U^c} h(z)\int_U G_U(x,y)\nu(y,z)\,\dy \dz\\
= &\int_U G_U(x,y) \int_{U^c} h(z) \nu(y,z)\, \dz\dy = \int_U G_U(x,y) \int_{\mR^d} (h(z)) - h(y)) \nu(y,z) \,\dz\dy\\
= &\int_U G_U(x,y) Lh(y)\,\dy.
\end{align*}
The usage of Fubini's theorem is justified as follows: for $z\in {\rm supp}\,h$ we have $|h(z)| P_D(x,z) \approx |h(z)| \nu(x,z)$, which is integrable by \eqref{es:nuSc1} and the integrability assumption on $u$.
\end{proof}
	
We next generalize the Hardy--Stein formula of  Bogdan, Dyda and Luks \cite[Lemma 3]{MR3251822}.
	\begin{lemma}\label{lem:BDL}
		Assume that {\rm \textbf{A1}} holds. If $u\colon\mathbb R^d\to\mathbb R$ is $L$-harmonic in $D$ and $U \subset\subset D$ is open and Lipschitz, then for every $x\in\mR^d$
		\begin{equation}\label{eq:lemBDL}
		\mE^x u^2(X_{\tau_U}) = u(x)^2 + \int_U G_U(x,y)\pord (u(z) - u(y))^2 \nu(z,y)\, \dz \dy.
		\end{equation}
	\end{lemma}
	\begin{proof}
Note that {\rm \textbf{A1}} yields \eqref{es:nuSc1}, which lets us use Lemma \ref{lem:dynk}.
As in Lemma \ref{lem:dynk} it suffices to consider $x\in U$.

\textsc{Case 1.} Assume that $\int_{U^c} u(z)^2 \nu(z,x)\, \dz = \infty.$ Then $\mE^x u(X_{\tau_U})^2 = \infty$ as well.
Indeed, take $\delta>0$ such that $B(x,2\delta)\subset U$. By domain monotonicity we have $P_U(x,z) \geq P_{B(x,\delta)}(x,z),$ $z\in U^c.$

For $z\in U^c$ we have $|z-x|\geq 2\delta$, hence
 $P_{B(x,\delta)}(x,z) \apprge \nu(x,z),$ cf. \cite[Lemma 2.2]{MR3729529}. It follows	that	 
		\begin{align*}
		\mE^x u^2(X_{\tau_U}) & = \int_{U^c} u(z)^2 P_U(x,z)\, \dz
\geq \int_{U^c} u(z)^2 P_{B(x,\delta)}(x,z)\, \dz\apprge \int_{U^c} u(z)^2 \nu(x,z) \,\dz = \infty.
		\end{align*}
		Further, we claim that in this case the right-hand side of \eqref{eq:lemBDL} is also infinite. Namely, we will check that $\pord (u(z) - u(y))^2 \nu(z,y)\, \dz=\infty$ for $y\in B(x,\frac \delta 2),$ and then use the positivity of the Green function $G_U$ in $U$ \cite{MR3729529}. If $z\in U^c$, then
$|z-y|\leq C_{\delta}|z-x|$ with some $C_{\delta}>0$. Therefore, by \eqref{es:nuSc1}, $\int_{U^c} u(z)^2 \nu(x,z) \dz = \infty$  yields $\int_{U^c} u(z)^2 \nu (y,z) \dz = \infty$. Since $u$ is continuous, the following sets are well-defined: $A = A(y) = \{z\in U^c: (u(y) - u(z))^2 \geq \frac 12 u(z)^2\}$.
 Note that $|u(z)|\leq C|u(y)|$  for $z\in U^c \setminus A$  and thus $\int_{U^c\setminus A} u(z)^2 \nu(x,z)\,\dz \lesssim u(y)^2 \nu(x,U^c) < \infty$. It follows that
$\int_A u(z)^2 \nu(y,z)\, \dz = \infty$, in particular $A(y)$ is nonempty. Therefore,
		\begin{align*}
		\infty & =
		\frac 12 \int_A u(z)^2 \nu(y,z)\, \dz \leq \int_A (u(y) - u(z))^2 \nu(y,z) \,\dz \leq \pord (u(z) - u(y))^2 \nu(y,z)\, \dz,
		\end{align*}
and the claim follows.

\textsc{Case 2.} Now assume that $\int_{U^c} u(y)^2 \nu(x,y)\, \dy <\infty$ for every $\eps>0$. Since $U\subset\subset D$, Theorem~\ref{lem:harmc2} implies that  $u^2\in C^2(\oU)$, hence $u^2\in L^1_{loc}(\mR^d)$. By Lemma \ref{lem:dynk}, $Lu^2$ is bounded in $\oU$ and
		\begin{equation}\label{eq:eu2}
		\mE^x u^2(X_{\tau_U}) = u(x)^2 + \int_U G_U(x,y) Lu^2(y)\, \dy.
		\end{equation}
 We can now compute $Lu^2(y)$ for $y\in U.$ The $L$-harmonicity of $u$ yields (cf. \cite[proof of Lemma 3]{MR3251822})
		\begin{align}\label{eq:lu2}
		Lu^2(y) = \pord (u(z) - u(y))^2  \nu(z,y) \,\dz, \quad y\in U.
		\end{align}
		The latter integral is convergent because $u$ is Lipschitz near $y,$ and far from $y$ we can use the assumed integrability condition.
		Inserting \eqref{eq:lu2} into \eqref{eq:eu2} yields the desired result.
	\end{proof}

\begin{proof}[Proof of Theorem~\ref{th:wazne}]
If $u$ is $L$-harmonic in $D$ and $x\in D$, then
		\begin{equation}\label{eq:BDL}
		\sup_{x\in U\subset\subset D} \mE^x u(X_{\tau_U})^2 = u(x)^2 + \int_D G_D(x,y)\pord (u(z) - u(y))^2 \nu(z,y) \,\dz \dy.
		\end{equation}
Indeed, for $x,y \in D$ we have $\sup_{U\subset\subset D}G_U(x,y)=G_D(x,y)$, hence \eqref{eq:BDL} follows from Lemma~\ref{lem:BDL} and the monotone convergence theorem. Clearly, \eqref{eq:BDL} is a reformulation \eqref{eq:BDL-1}.
\end{proof}

\subsection{Various notions of harmonicity}

\begin{definition}\label{def:weakdistr}
	{\rm (i)}	We say that $u\colon \mR^d \to \mR$ is weakly harmonic in $D$, if $u\in \mathcal{V}^D$ and $\ED(u,\phi) = 0$ for every $\phi\in \mathcal{V}^D_0$.
		
	{\rm (ii)}	We say that $u\in L^1_{loc}(\mR^d)$ is distributionally harmonic in $D$ if $\int_{\mR^d} u L\phi = 0$ for every $\phi \in C_c^\infty (D)$.
\end{definition}
The weak harmonicity implies the distributional harmonicity because of the following result.
\begin{lemma}\label{lem:weakdistr}
If \eqref{es:nuSc1} holds, $D$ is bounded, $u\in \mathcal{V}^D$, and $\phi \in C_c^\infty(D)$, then $\ED(u,\phi) = -\int_{\mR^d} u L\phi$.
\end{lemma}
\begin{proof}
	As in \cite[Lemma 3.3]{MR1825645} we obtain
		\begin{align*}
	-\int_{\mR^d} u L\phi = &\lim\limits_{\eps\to 0} \int_{\mR^d}u(x) \int_{
		|y-x|>\eps} (\phi(x) - \phi(y)) \nu(x,y)\, \dy \dx\\
	& = \lim\limits_{\epsilon\to 0}
	\int_{\mR^d}\phi(x) \int_{
		|y-x|>\eps
	} (u(x) - u(y)) \nu(x,y)\, \dy \dx\\
	&=\lim\limits_{\epsilon\to 0}
	\int_{\mR^d}\phi(y) \int_{
		|y-x|>\eps
	} (u(y) - u(x)) \nu(x,y)\, \dy \dx\\
	&=
	\frac12\lim\limits_{\epsilon\to 0}
	\iint_{
		|y-x|>\eps}
	(\phi(x)-\phi(y))
	(u(x) - u(y)) \nu(x,y)\, \dy \dx= \ED(u,\phi).
	\end{align*}
	We note that interchanging the limit and the integrations in the above calculations is justified and that the conditions  needed to extend the proof of \cite[Lemma 3.3]{MR1825645} to the present setting are satisfied.
Indeed, in the last line we may use the dominated convergence theorem, the inequality $2|ab|\le a^2+b^2$, and the fact that $u,\phi\in\mathcal{V}^D$.
Next, let $\epsilon > 0$. Arguing as in \cite[Proposition 1.2]{MR3738190}, for $x\in D$, we get that $|u(x) \int_{|x-y|>\eps} (\phi(x) - \phi(y))\nu(x,y) \,\dy|$ is bounded by $C|u(x)| \|\phi\|_{C^2} \in L^1(D)$. Furthermore if we let $U = {\rm supp}\, \phi$, then for $x \in D^c$ we have $|u(x) \int_{|x-y|> \eps} (\phi(x) - \phi(y))\nu(x,y)\, \dy| \leq 2\|\phi\|_{\infty} |u(x)|\nu(x,U)\,\dx \in  L^1(D^c)$, by Lemma \ref{lem:l1}. By the dominated convergence theorem,
	\begin{align*}
	-\int_{\mR^d} uL\phi
	&= \lim\limits_{\eps\to 0} \int_{\mR^d}u(x) \int_{
	|y-x|>\eps
	} (\phi(x) - \phi(y)) \nu(x,y)\, \dy \dx.
	\end{align*}
	Further, note that $\int_{\mR^d} |u\phi| < \infty$. By \eqref{es:nuSc1} and Lemma~\ref{lem:l1}, we get that
	$$\int_{\mR^d}\int_{|x-y|>\eps} |u(x)\phi(y)|\nu(x,y) \,\dx \dy \lesssim \int_{\mR^d} |\phi(y)| \,\dx \int_{\mR^d} |u(x)| (1\wedge \nu(0,x)) \,\dx< \infty.$$ Thus the assumptions of \cite[Lemma 3.3]{MR1825645} are satisfied. The proof is complete.

\end{proof}

\section{{
Solving the Dirichlet problem}}\label{sec:pfs}

\subsection{Sobolev regularity of Poisson integrals}
\begin{proof}[Proof of Theorem~\ref{th:toz}]
	{\rm (i)}
	Assume that $g\in L^1(P_D(x,\cdot))$ for some, hence for all $x\in D$, cf. the proof of Lemma \ref{lem:harmonic}. By Lemma \ref{lem:integrability} this is true if $\HD(g,g)<\infty$.We are going to prove that $\HD(g,g)=\mathcal{E}_D(u,u),$ where $u = P_D[g]$ is the Poisson extension of $g.$
By
\eqref{e:achm},
 $P_ D(x,z)\,\dz$ is a probability measure on $ D^c$ for every $x\in  D$. The integral of $g$
against this measure is equal to $u(x)$.
Recall that if $Y$ is a (real-valued) random variable
and $\mathbb E |Y|<\infty$, then for every $a\in\mathbb R$,
\begin{equation*}\label{eq:variances}
\mathbb E(Y-a)^2 =  \mathbb{E} (Y - \mathbb{E} Y)^2+\left(\mathbb EY-a\right)^2,
\end{equation*}
even if $\mathbb E Y^2=\infty$.
In particular, for $Y=u(X_{\tau_D})$
we have
$$\mathbb E^xY =u(x)\quad  \mathrm{and}\quad  \mathbb E^x Y^2=\int_{D^c}u^2(z)P_D(x,z)\,\dz,\quad x\in D.$$
Therefore,
$$\int_{ D^c}(u(w)-u(z))^2 P_D(x,z)\,\dz
= \int_{ D^c} (u(x)-u(z))^2 P_D(x,z)\,\dz + (u(w)-u(x))^2,\quad x\in D,\, w\in D^c,$$
and further, by Lemma \ref{lem:integrability},
 \begin{eqnarray}\label{eq:firstpart}
 2\HD(g,g)&=& \int_ D\int_{ D^c} \int_{ D^c} ((u(x)-u(z))^2P_ D(x,z)\nu(x,w)\,\dz\dw\dx \nonumber \\
 && +\int_ D\int_{ D^c} (u(w)-u(x))^2 \nu(x,w)\,\dw\dx \nonumber\\
 &=& \int_ D\int_{ D^c} (u(x)-u(z))^2 P_ D(x,z)\kappa_ D(x)\,\dz\dx + \int_ D\int_{ D^c} (u(w)-u(x))^2 \nu(x,w)\,\dw\dx.
  \end{eqnarray}
By \eqref{e:achm}, for $U^c$ satisfying VDC with $|\partial U| = 0$ we have
$$
\int_{U^c}(u(x)-u(z))^2P_U(x,z)\,\dz=
\mathbb E^x (u(X_{\tau_U})-u(x))^2,\quad x\in U.$$
By formula \eqref{eq:BDL}  applied to the function $z\mapsto \widetilde u(z)= u(z)-u(x)$,
for
each $x\in  D$ we get
 \begin{equation}\label{eq:aa}
 \sup_{x\in U\subset\subset D} \int_{U^c}(u(x)-u(z))^2P_U(x,z)\,\dz = \int_ D G_ D(x,y)\int_{\mathbb R^d} (u(y)-u(z))^2\nu(y,z)\,\dz\dy.
 \end{equation}

By Remark~\ref{rem:martingale}, $\{u(X_{\tau_U)}\}_{U\subset\subset  D}$ is a closed martingale.
Therefore the Hardy--Stein formula
\eqref{eq:BDL} remains valid if we replace  $\sup_{x\in U\subset\subset D}$ by $\lim_{x\in U\uparrow D}$ with $U\subset\subset D$ increasing to $D$. By \eqref{e:achm2}, for almost every trajectory of $X$, there exists $U\subset\subset D$ such that
$X_{\tau_U}=X_{\tau_D}$, so
$u(X_{\tau_U})\to u(X_{\tau_D})=g(X_{\tau_D})$ as $U\uparrow D$.
By the martingale convergence theorem
and \eqref{eq:aa},
 \begin{equation*}\label{eq:a}
\int_{ D^c}(u(x)-u(z))^2P_ D(x,z)\,\dz = \int_ D G_ D(x,y)\int_{\mathbb R^d} (u(y)-u(z))^2\nu(y,z)\,\dz\dy,\quad x\in D.
 \end{equation*}
We now turn our attention to the first integral in
\eqref{eq:firstpart}.
By Fubini--Tonelli,
 \begin{eqnarray*}
&& \int_ D\int_{ D^c} ((u(x)-u(z))^2 P_ D(x,z)\kappa_ D(x)\,\dz\dx\\
 &=&
  \int_ D\int_ D\int_{\mathbb R^d}G_ D(x,y) ((u(y)-u(z))^2\nu(y,z)\kappa_ D(x)\,\dz\dy\dx\\
  &=& \int_{\mathbb R^d}\int_ D \left[\int_ D \kappa_ D(x)G_ D(x,y)\dx\right] ((u(y)-u(z))^2 \nu(y,z)\,\dz\dy.
 \end{eqnarray*}
 Since $D^c$ satisfies VDC,
 by \eqref{eq:skad1}
 and Corollary \ref{c:achm},
\begin{equation}\label{eq:skad}
\int_ D \kappa_ D(x)G_ D(x,y)\,\dx = \mathbb P^y(X_{{\tau_ D}^-}\in  D)=1,\quad y\in  D.
\end{equation}
Therefore,
$$\int_ D\int_{ D^c} (u(x)-u(z))^2 P_ D(x,z)\kappa_ D(x)\,\dz\dx = \int_{\mathbb R^d}\int_ D (u(y)-u(z))^2 \nu(y,z)\,\dz\dy.$$
By this and \eqref{eq:firstpart}, we get part {\rm (i)} of the theorem:
\begin{eqnarray*}
2\HD(g,g)&=& \int_{\mathbb R^d}\int_ D (u(x)-u(y))^2 \nu(x,y)\,\dx \dy + \int_{ D}\int_{ D^c} (u(x)-u(y))^2\nu(x,y)\,\dx \dy
\\
&=& \iint\limits_{\mathbb R^d\times\mathbb R^d\setminus D^c\times  D^c}(u(x)-u(y))^2\nu(x,y)\,\dx \dy = 2\ED (u,u).
\end{eqnarray*}

\smallskip

\noindent {\rm (ii)}
Suppose that $D$ is bounded.
Assume that $u\in \vd$, i.e., $\ED(u,u) < \infty$.  Let $g=u|_{D^c}$.
Since $g$ has an extension in $\vd$, it also has a weakly harmonic extension \cite[Theorem 4.2]{MR3738190} with the same or smaller quadratic form  \cite[Lemma 4.8]{MR3738190}. Put differently, we may assume that $u$ is weakly harmonic, i.e., $\ED(u,\phi) = 0$ for every $\phi \in \mathcal{V}^D_0$. By Lemma \ref{lem:l1}, \ref{lem:weakdistr} and \cite[Theorem 1.1]{GKL}, after a modification on a set of Lebesgue measure zero, $u$ is $L$-harmonic in $D$, in particular for every Lipschitz $U\subset\subset D$ we have $P_U[|u|] < \infty$. Given that fact, the chain of identities from the proof
of part {\rm (i)} can be reversed with $D$ replaced by $U$, and we obtain $\mathcal{E}_U(u,u) = \mathcal{H}_U(u,u)$.
We then let
$U\uparrow D$.
Clearly, $\mathcal{E}_U(u,u) \uparrow \ED(u,u) <\infty$. By Fatou's lemma,
\begin{align*}
	\infty &>2\ED(u,u) =  \lim_{U\uparrow D} 2\mathcal{E}_U(u,u)
= \lim_{U\uparrow D} 2\mathcal H_U(u,u) \\
&\geq \iint_{\mR^d \times \mR^d} (u(z) - u(w))^2 \liminf_{U\uparrow D} \left(\gamma_U(z,w)\mathbf 1_{U\times U}(z,w)\right) \,\dz \dw \\
	&\geq \iint_{D\times D} (u(z) - u(w))^2 \int_{\mR^d} \nu(z,x)\liminf_{U\uparrow D}\left(\int_{\mR^d} G_U(x,y)\nu(y,w)\,\dy\right)\,\dx\dz\dw.
	\end{align*}
By the quasi-left continuity, $\mathbb{P}^x$--almost surely we have $\tau_U \uparrow \tau_D$ , cf. \cite[proof of Lemma~17]{MR1438304} and \cite[Theorem 40.12]{MR1739520}.
By the monotone convergence theorem,
	\begin{align*}
	\liminf_{U\uparrow D}\int_{\mR^d} G_U(x,y)\nu(y,w)\,\dy&= \liminf_{U\uparrow D}\mE^x\int_0^{\tau_U} \nu(X_t,w)\,\dt\\
	=\ &\mE^x \int_0^{\tau_D} \nu(X_t,w) \,\dt = \int_{\mR^d}G_D(x,y)\nu(y,w)\,\dy.
	\end{align*}
Thus, $\mathcal H_D(g,g)<\infty$, which
completes the proof for bounded sets $D$.

For unbounded $D$ we consider the nonempty intersections $D\cap B(0,R)$ for $R>0$. We have $\mathcal{E}_{D\cap B(0,R)}(u,u) < \infty$, so there exists a weakly harmonic $u_R$ such that $u_R = u$ a.e. on $D^c$. By \cite[Lemma 4.8]{MR3738190} and the above considerations $\mathcal E_D(u,u)\geq \mathcal{E}_{D\cap B(0,R)}(u,u) \geq \mathcal{E}_{D\cap B(0,R)}(u_R,u_R) = \mathcal{H}_{D\cap B(0,R)}(u,u)$ for all $R$.
We let $R\to\infty$ and by the monotone convergence theorem we get $\ED(u,u) \geq \HD(u,u)$.
\end{proof}

{
In the setting of {\rm Theorem~\ref{th:toz}} we immediately obtain the following
consequences.

\begin{corollary}\label{cor:ext}
$\mbox{\rm Ext}\, g = P_D[g]$ is a linear isometry from
$\xd$ into $\vd$ and ${\rm Tr}\, u= u|_{D^c}$ is a linear contraction from $\vd$ onto $\xd$. ${\rm Tr}\ {\rm Ext}$ is the identity operator on $\xd$ and ${\rm Ext} {\rm Tr}$ is a contraction on $\vd$. Furthermore ${\rm Tr}\  u=0$ characterizes $u\in \vd_0$.
\end{corollary}
Thus the Poisson integral and the restriction to $D^c$ may serve as the extension and trace between the Sobolev spaces $\vd$ and $\xd$, correspondingly. For the many uses of analogous operators in PDEs we refer the reader to \cite{DINEZZA2012521,MR2597943}.

\begin{corollary}\label{cor:Di1}
{If $P_D[|u|]<\infty$ on $D$, in particular, if $\mathcal{E}_{\mathbb R^d}(u,u)<\infty$, then}
\begin{equation*}
\tfrac12\!\!\!\iint\limits_{ D^c\times  D^c} (u(w)-u(z))^2\left(\gamma_ D(w,z)+\nu(w,z)\right)\,\dw\dz =
\mathcal{E}_{\mathbb R^d}(P_D[u],P_D[u]) \le \mathcal{E}_{\mathbb R^d}(u,u).
\end{equation*}
\end{corollary}

Corollary~\ref{cor:Di1} and the Douglas identity in Theorem~\ref{th:toz} may be considered as
analogues of the Douglas integral \cite[(1.2.18)]{MR2778606}.

\begin{corollary}
		Suppose $L$ satisfies {\bf A1, A2}. Let a nonempty open bounded set $D\subset\mathbb R^d$ have continuous boundary and let {$D^c$} fulfill VDC.
		Then the non-homogeneous Dirichlet problem  \eqref{eq:nhDp} has a unique solution for arbitrary $g\in\xd$ and $f\in L^2(D)$.
\end{corollary}

In the next subsection we get the minimality of the Poisson extension for $\ED$ (and $\mathcal{E}_{\mathbb R^d}$), which lets us represent weakly harmonic functions as the Poisson integrals of the exterior data.

\subsection{Weak and variational solutions}
\label{sec:variational}
\noindent The next proposition shows that weak solutions coincide with the variational solutions of \eqref{eq:Dp}. The proof is classical but we include it here to make our argument self-contained, cf. \cite{MR3447732,MR3738190}.
\begin{proposition}\label{prop:minimal}
	Let $g\in\xd$ and let $u$ be a weak solution of \eqref{eq:Dp}.  If $\overline{g}\colon \mathbb R^d\to\mathbb R$ is another
	measurable func\-tion  equal to $g$ $a.e.$ on $D^c,$ then
	\begin{equation}\label{eq:toprove}
	\ED (u,u)\leq \ED (\overline{g}, \overline{g}).
	\end{equation}
	The converse is also true.
\end{proposition}

\begin{proof}
	Note that \eqref{eq:toprove} holds trivially when either
	$\ED (\overline{g}, \overline{g})=+\infty$ or $\ED (u,u)=0$. Therefore we may assume otherwise.
	We have $\overline g-u\in \vd_0.$
	Since $u$ is a weak solution,
	$\ED (u,\overline{g}-u)=0$, hence
	$\ED (\overline{g},u)=\ED (u,u)$ and
	\begin{equation*}\ED (u,u)=\ED (\overline{g},u)\leq \left(\ED (\overline{g},\overline{g})\right)^{1/2}\left(\ED (u,u)\right)^{1/2}.
	\end{equation*}
	Canceling out $\ED (u,u)^{1/2}>0$,
	we obtain \eqref{eq:toprove}.
	
	For the converse, let $\phi \in \vd_0$. Since $u$ is a minimizer, we have
	\begin{equation*}
	0\leq \ED(u+\lambda\phi,u+\lambda\phi) - \ED(u,u) = 2\lambda\ED(u,\phi) + \lambda^2\ED(\phi,\phi),\quad \lambda \in \mR.
	\end{equation*}
	This necessitates that $\ED(u,\phi) = 0$, hence $u$ is a weak solution.
\end{proof}
By Theorem \ref{th:toz}, if $g\in \xd,$ then its Poisson extension belongs to $\vd$. In fact, the Poisson extension $P_D[g]$ is the weak solution of \eqref{eq:Dp}, as we will shortly see.

 \begin{theorem}\label{th:var}
 Suppose $L$ satisfies {\bf A1, A2}. Let nonempty open $D\subset\mathbb R^d$ have continuous boundary and let {$D^c$} fulfill VDC. If  $g\in\xd,$ then $u=P_D[g]$ is a solution to the Dirichlet problem \eqref{eq:Dp}.
 Furthermore, for bounded $D$ the solution is unique.
 \end{theorem}

\begin{proof}
	The uniqueness for bounded $D$ follows from \cite[Theorem 4.2]{MR3738190}, therefore it is enough to show that $P_D[g]$ is a weak solution.
	
If $g\in\xd,$ then $u=P_D[g]$ is well-defined, $u=g$ on $D^c$ and by
Theorem \ref{th:toz},
$u\in \vd.$
By Theorem \ref{th:density}
we only need to verify \eqref{eq:ws-a} for
$\phi\in C_c^{\infty}(D)$.
Let
$\epsilon>0$, $\nu_\epsilon(x,y)=\nu(x,y)\mathbf 1_{|x-y|>\epsilon}$ and
  \begin{eqnarray*}2\ED^\epsilon (u,\phi)&=&\iint\limits_{(\mathbb R^d\times\mathbb R^d\setminus D^c\times D^c) \cap \{|x-y|>\epsilon\}} (u(x)-u(y))(\phi(x)-\phi(y))\nu(x,y)\,\dx \dy\\
  &=& \iint\limits_{\mathbb R^d\times\mathbb R^d\setminus D^c\times D^c} (u(x)-u(y))(\phi(x)-\phi(y))\nu_\epsilon(x,y)\,\dx \dy.
  \end{eqnarray*}
  Since $u,\phi\in \vd,$ the integral $\ED(u,\phi)$ is absolutely convergent and $\ED^\epsilon(u,\phi)\to\ED (u,\phi)$ as $\epsilon\to 0$.
We claim that $\ED^\epsilon (u, \phi)\to 0$ when $\epsilon \to 0.$
Indeed, the integral
\[\iint\limits_{\mathcal A} (u(x)-u(y))(\phi(x)-\phi(y))\nu(x,y)\,\dx\dy\]
 is absolutely convergent for every set $\mathcal A\subset \mathbb R^d\times\mathbb R^d\setminus  D^c\times  D^c$, hence
\begin{eqnarray*}
  2\ED^\epsilon(u,\phi) &=& \int_ D\int_{ D}(u(x)-u(y))( \phi(x)- \phi(y))\nu_\epsilon(x,y)\,\dx \dy\\
  &&+\int_ D\int_{ D^c}(u(x)-u(y))(\phi(x)-\phi(y))\nu_\epsilon(x,y)\,\dx \dy\\
  &&+ \int_{ D^c}\int_{  D}(u(x)-u(y))( \phi(x)- \phi(y))\nu_\epsilon(x,y)\,\dx \dy =: {\rm I}+{\rm I\!I}+{\rm I\!I\!I}.
  \end{eqnarray*}
By the symmetry of $\nu_\epsilon$  and the fact that $\phi\equiv 0$ on $D^c,$ we readily see that ${\rm I\!I}={\rm I\!I\!I}=\iint\limits_{D^c\times D} (u(x)-u(y))\phi(x)\nu_\epsilon(x,y)\,\dx\dy.$
 Also,
 ${\rm I} = 2 \iint\limits_{D\times D} ( u(x) - u(y) )\phi(x)\nu_\epsilon(x,y)\,\dx \dy$, which converges absolutely
  by the Cauchy--Schwarz inequality and the fact that $\phi(x)\nu_\epsilon(x,y)\,\dx\dy$ is a finite measure.
 Thus,
 \begin{eqnarray*}
 2\ED^\epsilon(u,\phi) &= &2\int_{\mathbb R^d}\int_ D (u(x)-u(y))\phi(x)\nu_\epsilon(x,y)\,\dx\dy\\
 &=&2\int_ D \phi(x)\left(\int_{\mathbb R^d} (u(x)-u(y))\nu_\epsilon(x,y)\,\dy\right)\,\dx \\
 &=& - 2\int_ D \phi(x) L_\epsilon u(x)\,\dx= -2\int_{{\rm \small supp }\, \phi} \phi(x)L_\epsilon u(x)\,\dx,
 \end{eqnarray*}
 where
 \[L_\epsilon u(x) := \int_{\mathbb R^d}(u(y)-u(x))\nu_\epsilon(x,y)\,\dy.\]
 The function $u$ is regular $L$-harmonic (see Definition \ref{def:harm}). By Theorem~\ref{lem:harmc2}, $u\in C^2(D)$ and by Lemma \ref{lem:harmlu}, $Lu(x)=0$ for $x\in D$.
 In particular, $L_\epsilon u(x)\to Lu(x)=0$ for $x\in D.$
 We will prove  that the convergence is uniform on the support of~$\phi.$
For $x\in D$, $0<\eta<\epsilon$,
 \begin{equation}\label{eq:difference}
 L_\eta u(x)-L_\epsilon u(x)=
 \int_{\eta<|x-y|\leq \epsilon}(u(y)-u(x))\nu(x,y)\,\dy.
 \end{equation}
Let $\delta=\mbox{dist}(\mbox{supp}\,\phi,  D^c)>0$ and let $\epsilon<\delta/2$.  If $x\in\mbox{supp}\,\phi,$
then points $y$ appearing in \eqref{eq:difference} belong to the compact set $K:=\mbox{supp}\,\phi+\overline{B}_{\delta / 2}\subset  D$.
Since $u\in C^2(D)$,
$$
C_K:=\sup_{x\in K, 1\leq i,j\leq d} (|u(x)|,\left|{\partial^i u(x)}\right|,\left|\partial^{ij} u(x)\right|)<\infty.$$
By the symmetry $\nu(x,y)=\nu(y,x)$,
\begin{align*}
 &L_\eta u(x)-L_\epsilon u(x)=
 \int_{\eta<|x-y|\leq \epsilon}(u(y)-u(x))\nu(x,y)\,\dy\\
 &=
 \int_{\eta<|x-y|\leq \epsilon}(u(y)-u(x) - \nabla u(x)\cdot (y-x))\nu(x,y)\,\dy .
\end{align*}
From Taylor's formula we obtain
\begin{align*}
|Lu(x)-L_\epsilon u(x)|&\leq \frac{C_K}{2} \lim\limits_{\eta\to 0} \int_{\eta<|x-y|\leq \epsilon} |y-x|^2 \nu(x,y)\,\dy \\
&\leq  \frac{C_K}{2}  \int_{|x-y|\leq \epsilon} |y-x|^2 \nu(x,y)\,\dy
= C_K C_\eps, \quad x\in {\rm\small supp}\, \phi,
 \end{align*}
 and $C_\eps \to 0$ as $\eps \to 0$, so
 $\int
 \phi(x)L_\epsilon u(x)\,\dx \to \int_{{\rm\small supp}\,\phi}
 \phi(x)Lu(x)\,\dx=0$.
Our claim follows: $\ED^\epsilon(u,\phi)\to -2\int_{{\rm\small supp}\,\phi} \phi(x)Lu(x)\,\dx=0$. Therefore $\ED (u,\phi)=0$, as needed.
\end{proof}

\subsection{Equivalence of definitions of harmonicity} The definitions of harmonicity can be unified. We say that $\tilde u$ is a modification of $u$ if $\tilde u=u$ $a.s.$
\begin{theorem}\label{th:equivalence}
	Let $u\in \mathcal{V}^D$ and let $D$ be bounded with continuous boundary. Under the assumptions of Theorem \ref{th:toz} the following statements are equivalent:
	\begin{enumerate}
		\item[(i)] $\int_{\mR^d} u L\phi = 0$ for every $\phi \in C_c^\infty(D)$ (distributionally harmonic),
		\item[(ii)] $\ED(u,\phi) = 0$ for every $\phi \in \mathcal{V}^D_0$ (weakly harmonic),
		\item[(iii)] $u$ has a modification that is $L$-harmonic,
		\item[(iv)] $u$ has a modification that is regular $L$-harmonic, and $u=P_D[u]$ \textit{a.e}.
	\end{enumerate}
Furthermore, any of the statements above yields $Lu(x) = 0$ in $D$.
\end{theorem}
\begin{proof}
First, (iv) implies (iii) by Lemma \ref{lem:harmonic}. Then we prove that (iii) implies (ii). Indeed, by Theorem \ref{th:var} used for Lipschitz open $U\subset\subset D$ we get $\mathcal{E}_U(u,\phi) = 0$ for every $\phi \in C_c^\infty(U)$. By the dominated convergence theorem and the fact that $u\in \mathcal{V}^D$, we have $\mathcal{E}_D(u,\phi) = 0$. Since for every $\phi\in C_c^\infty(D)$ there exists $U\subset \subset D$ containing the support of $\phi$, we get that $\mathcal{E}_D(u,\phi) = 0$ for every $\phi\in C_c^\infty(D)$. Then we use the density of smooth functions in $\mathcal V^D_0,$ see Theorem \ref{th:density}. Then (ii) implies (i) by Lemma \ref{lem:weakdistr}. Finally, (i) implies (iv). Indeed, by \cite[Theorem 1.1]{GKL} $u$ is harmonic and thus weakly harmonic. By the trace theorem ${\rm Tr}\, u \in \xd$, and by the extension theorem $P_D[u]\in \mathcal{V}^D$. By Theorem \ref{th:var}, $P_D[u]$ is the unique weakly harmonic function equal to $u$ \textit{a.e.} on $D^c$. Hence $u = P_D[u]$ \textit{a.e.} on $\mR^d$.
The statement	$Lu=0$ follows from Lemma \ref{lem:harmlu}.
\end{proof}

Theorem \ref{lem:harmc2} allows for the following extension which may be regarded as
a counterpart of the Weyl's lemma for nonlocal operators.
	\begin{theorem}
		Assume that for every $r_0> 0$ there exists $C(r_0)$ such that $|\nu^{(k)}(r)| \lesssim C(r_0)\nu(r)$ for $r>r_0$ and $k=1,\ldots, n$. If $u\in L^1(1\wedge \nu)$ is distributionally harmonic in $D$, then $u\in C^n(D)$.
	\end{theorem}
\begin{proof}
Adapt the proof of Theorem \ref{lem:harmc2}, starting from \eqref{eq:poissexp}.
\end{proof}

{
\section{Estimates of the interaction kernel}
\label{sec:ikg}
In this section we prove sharp estimates of $\gamma_D$
for
the half-space and for bounded $C^{1,1}$ open sets.
In the proof of the result for the half-space we often use the following global scalings.
\begin{description}
	\item[A4]There exist constants $\alpha,\beta\in (0,2)$ and $c,C>0$ such that
	\begin{equation}\label{eq:a2prim}
	\nu(\lambda r)\leq C \lambda^{-d-\beta}\nu(r) ,\qquad 0<\lambda\leq 1,\, r>0.
	\end{equation}
	\begin{equation}\label{eq:scalNua}\nu(\lambda r) \geq c\lambda^{-d-\alpha}\nu(r),\quad\quad\, 0<\lambda\leq 1,\, r>0.
	\end{equation}
\end{description}
Note that \eqref{eq:a2prim} is but a global version of \eqref{eq:nuSc},
equivalent to $r^{d+\beta}\nu(r)$ being almost increasing on $(0,\infty)$: $p^{d+\beta}\nu(p) \leq C r^{d+\beta}\nu(r)$ if $0<p<r<\infty$, cf. \cite[Section~3]{MR3165234}.
Clearly, {\bf A4} holds true if $L=\Delta^{\alpha/2}$.

We start with some basic observations.
If \eqref{eq:a2prim} holds, then by integrating \eqref{eq:Kdef} in polar coordinates and by changing variables (cf. \eqref{e:cKnu}) we get
\begin{equation}\label{eq:compar}
\nu(s) \approx \frac{K(s)}{s^d},\quad s>0.
\end{equation}
For $a\in(0,2]$ we denote
$$U_a(s)= \frac{K(s)}{(h(s))^as^d}, \quad s>0.$$
Due to \cite[Theorem 1.2]{MR3729529}, unimodality of $\nu$ and
\eqref{eq:compar}, $U_a$ is almost decreasing,  i.e., there is a constant $c_a>0$ such that for all $0<s_1<s_2$ we have $U_a(s_1)\geq c_a U_a(s_2)$, in short $U_a(s_1)\apprge U_a(s_2)$.
It is known \cite[(3.5)]{MR3350043} that $h'(r)=-2K(r)/r$. In particular, $h$ is decreasing and  $V$ is increasing, see \eqref{eq:Vdef}. A~direct calculation gives
\begin{equation}\label{eq:Vprim}
-\Big(\frac{1}{V(s)}\Big)' = \frac{V(s)K(s)}{s} \approx V(s)\nu(s)s^{d-1},\qquad [V^2]'(s) = 2s^{d-1}U_2(s).
\end{equation}
The
factor $s^{d-1}$
will be useful for integrations in polar coordinates.
It is also easy to verify that
$s^2h(s)$ is nondecreasing, hence $V(s)/s$ is nonincreasing and for every $\lambda\in(0,1)$ we have
\begin{equation}\label{eq:VV}
V(s)\geq V(\lambda s)\geq \lambda V(s),\quad s>0.
\end{equation}
Here is our main result
for the half-space
$$H=\{x\in\mathbb R^d: x_d>0\}.$$
\begin{theorem}\label{prop:gammaHalfspace}
	Let $d\geq 3$ and assume that \eqref{eq:a2prim} holds true. Then,
	$$\gamma_{H}(z,w)\leq C \frac{V^2(r(z,w))\nu(r(z,w))}{V(\delta_H(z))V(\delta_H(w))}.$$
	If we additionally assume \eqref{eq:scalNua},
	then
	$$\gamma_{H}(z,w)\approx \frac{V^2(r(z,w))\nu(r(z,w))}{V(\delta_H(z))V(\delta_H(w))}.$$
	Here $r(z,w) = \delta_H(z) + \delta_H(w) + |z-w|$ and $\delta_H(z) = \dist (z,\partial H)$.
\end{theorem}
The proof of Theorem~\ref{prop:gammaHalfspace} (given below) uses the following lemma.
\begin{lemma}\label{lem:PoissonEst}
	Let $d\geq 3.$ Assume that \eqref{eq:a2prim} holds true.
	 Then,
	$$P_H(x,z)\approx \frac{V(\delta_H(x))}{V(\delta_H(z))}V^2(|x-z|)\nu(|x-z|),\quad x\in H,\,\, z\in \overline{H}^c.$$
\end{lemma}
\begin{proof}
	 By \cite[Theorem 1.13]{MR3729529},
	\begin{equation}\label{eq:green-est}
G_H(x,y)\approx \frac{V(\delta_H(x))}{V(\delta_H(x)+|x-y|)}\frac{V(\delta_H(y))}{V(\delta_H(y)+|x-y|)}U_2(|x-y|),\quad x,y\in H.
\end{equation}
From \eqref{eq:a2prim}, the Ikeda--Watanabe formula \eqref{eq:ikedawatanabe} and the monotonicity properties of $V, U_a, \nu,$
	\begin{align*}
	\frac{P_H(x,z)}{V(\delta_H(x))}\lesssim \ &\int_{H\cap\{|x-z|\leq 2 |x-y|\}}\frac{V(\delta_H(y))}{V^2(|x-z|/2)}U_2(|x-z|/2)\nu(z,y)\,\dy \\
	&+\int_{|x-z|> 2 |x-y|}\frac{1}{V(|x-y|)}U_2(|x-y|)\nu((x-z)/2)\,\dy\\
	\leq\ &U_1(|x-z|/2)\int_{\overline{B}^c(z,\delta_H(z))}V(|y-z|)\nu(z,y)\,\dy\\
	&+ \nu(x,z)\int_{B_{|x-z|/2}}U_{3/2}(|y|)\,\dy\\ \leq \
& U_1(|x-z|/2)\int_{\overline{B}_{\delta_H(z)}^c} V(|y|)\nu(|y|)\dy  + \nu(x,z)\int_{B_{|x-z|/2}}U_{3/2}(|y|)\dy\\
	\lesssim &\frac{U_1(|x-z|)}{V(\delta_H(z))}+U_{1/2}(|x-z|) = 2\frac{U_1(|x-z|)}{V(\delta_H(z))}.
	\end{align*}
	In the last inequality we use (\ref{eq:compar}) and the formula $h'(r)=-2K(r)/r,$ which result in
\begin{eqnarray*}
\int_r^\infty\frac{K(s)}{h^{1/2}(s)s}\,\ds=h^{1/2}(r)
	&\mbox{and} &\int^r_0\frac{K(s)}{sh^{3/2}(s)}\,\ds=\frac{1}{h^{1/2}(r)}.
\end{eqnarray*}
We next prove a matching  lower estimate.
Using repeatedly the monotonicity properties of $U_a, V,$ the inequality \eqref{eq:VV} and the scaling of $\nu$ we see that
up to a multiplicative constant, $P_H(x,z)/V(\delta_H(x))$ is not less than
	\begin{eqnarray*}
	&&\int_{H\cap\{|y-z|\leq 2|x-z|\}}\frac{V(\delta_H(y))}{V^2(5|x-z|)}U_2(3|x-z|)\nu(z,y)\,\dy \\&&+ \int_{H\cap\{|y-x|\leq 2|x-z|\}}\frac{V(\delta_H(y))}{V(\delta_H(y)+|x-y|)V(3|x-z|)}
U_2(|x-y|)\nu(3|x-z|)\,\dy\\
	&\gtrsim& U_1(5|x-z|){\rm I} + \frac{\nu(|x-z|)}{V(3|x-z|)}{\rm I\!I}\ \gtrsim
 \ U_1(|x-z|)\left({\rm I}+\frac{1}{V^3(2|x-z|)}{\rm I\!I} \right),
\end{eqnarray*}
where
	\begin{align*}
	{\rm I}=&\int_{H\cap \{|y-z|\leq 2|x-z|\}}V(\delta_H(y))\nu(z,y)\,\dy,\\
	{\rm I\!I}=&\int_{H\cap \{|y-x|\leq 2|x-z|\}}\frac{V(\delta_H(y))}{V(\delta_H(y)+|x-y|)}U_2(|x-y|)\,\dy.\\
	\end{align*}
First we estimate the integral ${\rm I}$. Without loss of generality we may and do assume that $z=(0,\ldots,0,z_d)$ with $z_d<0$. Let $\Gamma=\{(\tilde{y},y_d):|\tilde{y}|<y_d\}$. Then, for $y\in \Gamma$, we have $2\delta_H(y)\geq |y-z|-\delta_H(z)$.  Hence, by the rotational invariance of $\nu$  and \eqref{eq:VV} we obtain
	\begin{align*}{\rm I \ \ }&\geq \int_{\Gamma\cap \{|y-z|\leq 2|x-z|\}} V((|y-z|-\delta_H(z))/2)\nu(|y-z|)\,\dy\\&\geq c(d) \int_{3\delta_H(z)/2 \leq |y-z|\leq 2|x-z|}V(|y-z|-\delta_H(z))\nu(y,z)\,\dy\\
	&\gtrsim \int^{2|x-z|}_{3\delta_H(z)/2}V({s})\nu(s)s^{d-1}\ds\approx \frac{1}{V({3\delta_H(z)/2})}-\frac{1}{V(2|x-z|)}.
	\end{align*}
Similarly,
	\begin{align*}{\rm I\!I\ \ }\ &\geq   \int_{|y-x|\leq 2(|x-z|\wedge y_d)}\frac{V(\delta_H(y))}{V(3\delta_H(y))}U_2(|x-y|)\,\dy\\
	&{\gtrsim}\int_{|y-x|\leq 2|x-z|}U_2(|x-y|)\,\dy\approx  V^2(2|x-z|),
	\end{align*}
	where in the second inequality we use the isotropy of $U_2$ and the inclusion
	$$\{y:|y-x|\leq 2y_d\}\supset\{y:|y-x|\leq 2(y_d-x_d)_+\}\supset x+\Gamma.$$
	Hence, up to a multiplicative constant, $P_H(x,z)/V(\delta_H(x))$ is not less than
\begin{eqnarray*}
&& U_1(|x-z|)\left(\frac{1}{V(3\delta_H(z)/2)}-\frac{1}{V(2|x-z|)}+\frac{1}{V(2|x-z|)}\right)
\geq \frac{U_1(|x-z|)}{V(\delta_H(z))} .
\end{eqnarray*}
	Since $U_1(s)\approx \nu(s) V^2(s),$ the proof is complete.
\end{proof}
\begin{proof}[Proof of Theorem~\ref{prop:gammaHalfspace}]
	We have
	$$\gamma_H(z,w)\approx \frac{1}{V(\delta_H(z))}\int_{H}V(\delta_H(x))V^2(|x-z|)
\nu(|z-x|)\nu(|w-x|)\,\dx.$$
	Let $\tilde{z}\in H$ be the reflection of $z$ in the hyperplane $\{x_d=0\}$.
	Then $|w-\tilde{z}|\approx r(z,w)$ and for $x\in H$ we have $|x-\tilde{z}|<|x-z|,$ and $\delta_H(\tilde{z}),\delta_H(x)\leq |x-z|.$ Consequently, the estimates of the Green function \eqref{eq:green-est} and Lemma \ref{lem:PoissonEst} imply
	 \begin{eqnarray}\label{eq:deltaz_0}\gamma_H(z,w)V^2(\delta_H(\tilde{z}))&\approx& \int_H \frac{V(\delta_H(x))V(\delta_H(\tilde{z}))}
{V^2(\delta_H(\tilde{z})+|x-z|)}V^4(|x-z|)\nu(|z-x|)\nu(|w-x|)\,\dx\\
	&\lesssim& \ P_H(\tilde{z},w)\approx \frac{V(\delta_H(\tilde{z}))}{V(\delta_H(w))}V^2(|\tilde{z}-w|)
\nu(|\tilde{z}-w|).\label{eq:ineqpois}
	\end{eqnarray}
We next assume \eqref{eq:scalNua} and prove  the matching lower bound.  It suffices to replace $z$ with $\tilde z$ in the right-hand side of \eqref{eq:deltaz_0} because then we have approximation $\approx$ instead of inequality $\lesssim$ in \eqref{eq:ineqpois}. To this end we again use \eqref{eq:VV} and obtain
\begin{align}&\int_{B(\tilde{z},\delta_H(z)/2)}
\frac{V(\delta_H(x))V(\delta_H(\tilde{z}))}
{V^2(\delta_H(\tilde{z})+|x-z|)}V^4(|x-z|)\nu(|z-x|)\nu(|w-x|)\,
\dx\nonumber\\
	\approx & \ \ V^4(\delta_H(z))\nu(\delta_H(z))\nu(|w-\tilde{z}|)
\delta_H(z)^d.\label{eq:deltaz1}
\end{align}
	For the integrand with $\tilde z$ we  have \begin{align}&\int_{B(\tilde{z},\delta_H(z)/2)}\frac{V(\delta_H(x))V(\delta_H(\tilde{z}))}{V^2(\delta_H(\tilde{z})+|x-\tilde{z}|)}V^4(|x-\tilde{z}|)\nu(\tilde{z}-x)\nu(w,x)\,\dx\nonumber\\
	&\approx \nu(|w-\tilde{z}|)\int_{B_{\delta_H(z)/2}}V^4(|x|)\nu(|x|)\,\dx
	\approx V^2(\delta_H(z))\nu(r(z,w)).\label{eq:deltaz2}
	\end{align}
The last comparison follows from
$V^4(s)\nu(s)s^{d-1}\approx [V^2]'(s).$ Since \eqref{eq:scalNua} gives  $\nu(r)r^d V^2(r)\approx 1$, the right-hand sides of \eqref{eq:deltaz1} and \eqref{eq:deltaz2} are comparable.  We have $|x-\tilde{z}|\approx|x-z|$, for $x\in H$ such that $|x-\tilde{z}|\geq \delta_H(z)/2$. Therefore we can replace $z$ by $\tilde{z}$ in the integrand in \eqref{eq:deltaz_0}, and so
	$$\gamma_H(z,w)\approx \frac{P_H(\tilde{z},w)}{V^2(\delta_H(\tilde{z}))}\approx \frac{V^2(r(z,w))}{V(\delta_H(z))V(\delta_H(w))}\nu(r(z,w)).$$
	
\end{proof}

\medskip
The result for the bounded $C^{1,1}$ open sets has a similar proof, so we will be brief.
\begin{proof}[Proof of Theorem \ref{th:gamma-est}] Let $D$ be $C^{1,1}$ at scale $R >0$.
\\ {\rm (i)}
First we let $\delta_D(z),\delta_D(w)\geq  R$.  Since $$\int_DG_D(x,y)\,\dy=\mE^x \tau_D,$$ by the radial monotonicity of $\nu$ we get
	$$ \nu(\delta_D(w)+\diam(D))\mE^x \tau_D \leq P_D(x,w)\leq \nu(\delta_D(w))\mE^x \tau_D.$$
	By \eqref{es:nuSc1},
	$$ \nu\left(\delta_D(w)+\diam(D)\right)\approx \nu(\delta_D(w)).$$
	These imply
	$$\gamma_D(z,w)\approx \nu(\delta_D(z))\nu(\delta_D(w))\int_{D}\mE^x\tau_D\, \dx,$$
which ends the proof in the first case.

\noindent {\rm (ii)}	We next assume that $\delta_D(z)\leq  R \leq \delta_D(w)$. We get
	$$\gamma_D(z,w)\approx \nu(\delta_D(w))\int_{D}\mE^x\tau_D \nu(z,x)\,\dx.$$
	Let $A=B(z,2\diam(D))\setminus \overline{B(z,\delta_D(z))}.$ By \cite[Theorem 4.6 and Proposition 5.2]{MR3350043},
	\begin{align*}\int_{D}\mE^x\tau_D \nu(z,x)\,\dx&\leq \int_A\mE^x\tau_A\nu(z,x)\,\dx\leq c_1 \int_A V(\delta_A(x))\nu(z,x)\,\dx\\ & \leq c_1\int_{\overline{B(0,\delta_D(z))}^c}V(|y|)\nu(y)\,\dy.
	\end{align*}
	Using \cite[Lemma 3.5]{MR3350043} we obtain
	$$\gamma_D(z,w)\leq c \nu(\delta_D(w))\frac{1}{V(\delta_D(z))}.$$
	Since $D$ is $C^{1,1}$, there is $x_0\in D$ such that $B=B(x_0,R)\subset D$.
	By \cite[Theorem 2.6]{MR3104097},
	\begin{align*}&\int_{D}\mE^x\tau_D \nu(z,x)\,\dx=  \int_{D}P_D(x,z) \,\dx\\
	&\geq c_2\int_{B(x_0,R/2)} \frac{V(\delta_D(x))}{V(\delta_D(z))}V^2(|x-z|)\nu(x,z)\, \dx\geq c_2 (R/\diam(D))^d\frac{1}{V(\delta_D(z))}.
	\end{align*}
Therefore,
	$$\gamma_D(z,w)\approx \nu(\delta_D(w))\frac{1}{V(\delta_D(z))}.$$

\noindent (iii) Finally, let
$\delta_D(z),\delta_D(w)<R.$
By \cite[Proposition 4.4 and  Theorem 4.5]{MR3249349}
the Dirichlet heat kernel of $D$ satisfies
\begin{equation*}
p_D(t,x,y)\approx \me^{-\lambda(D) t} \left(\frac{V(\delta_D(x))}{\sqrt{t}\wedge V(r)}\wedge 1\right)\left(\frac{V(\delta_D(x))}{\sqrt{t}\wedge V(r)}\wedge 1\right)p\left(t\wedge V^2(r),x,y\right),\quad t>0,\,x,y\in D,
\end{equation*}
where $\lambda(D)\approx 1/V^2(R)$.
Integrating against time we get rather standard estimates of the Green function, cf. \cite[proof of Theorem 7.3]{MR3237737}. For instance if $d\geq 2$, then
\begin{equation*}
G_D(x,y) \approx U_2(|x-y|)\bigg(\frac{V(\delta_D(x))V(\delta_D(y))}{V^2(|x-y|)}\wedge 1\bigg),\quad x,y\in D.
\end{equation*}
The Ikeda--Watanabe formula yields estimates for the Poisson kernel, cf. \cite[Theorem 2.6]{MR3104097},
\begin{equation}
P_D(x,z)\approx\frac{V(\delta_D(x))}{V(\delta_D(z))}\frac{1}{|x-z|^d},\quad x\in D,\, \delta_D(z)<R.
\end{equation}
By similar calculation as in the proof of Theorem~\ref{prop:gammaHalfspace} we obtain the estimate  in Theorem \ref{th:gamma-est}.
		\end{proof}

\section{Examples}\label{sec:examples}
In this section we provide examples of L\'evy measures other than \eqref{eq:lm} which satisfy {\rm \textbf{A1}} and {\rm \textbf{A2}}.

\begin{example}
{\rm
By inspection, {\rm \textbf{A1}} and {\rm \textbf{A2}}
are satisfied
when the L\'{e}vy density is
$$\nu(z)=\frac{1}{|z|^d\ln(2+|z|)},\quad z\in \Rd.$$
Due to mild singularity of $\nu$ at the origin the resulting operator $L$ may be considered
of "0-order".
}
\end{example}

For the rest of this section we consider
$\nu$ given by
 \begin{equation}\label{eq:SBM}\nu(r) = \int_0^\infty g_t(r) \,\eta(\dt),\quad r>0.
 \end{equation}
Here $g_t(r)=(4\pi)^{-d/2}\exp(-r^2/(4t))$,
and $\eta$ is the L\'evy measure of a nontrivial jump subordinator, i.e., $\eta$ is a
measure  on the real line such that
$\eta((-\infty,0]) = 0$ and
$0<\int_0^\infty (1\wedge t)\, \eta(\dt) < \infty.$ Note that $\nu$ is strictly positive and decreasing.

Let $$\varphi(\lambda)=\int_0^\infty(1-\me^{-\lambda t})\,\eta(\dt),\quad \lambda\geq 0.$$
The function is nonnegative and its derivative is completely monotone, i.e., it is a Bernstein function.
The L\'evy--Khinchine exponent corresponding to $\nu$ is $$
\psi(\xi)=\varphi(|\xi|^2),\quad \xi \in \Rd,
$$
and $L=-\varphi(-\Delta)$.
The corresponding L\'evy proces is called the subordinate Brownian motion.
Furthermore,
$\varphi$
is a {\em complete} Bernstein function if
\[\varphi(\lambda)=
\int_0^\infty (1-{\me}^{-t\lambda}) f(t)\,\dt,\quad \lambda\geq 0,\]
with a completely monotone $f$. See
Schilling, Song, and Vondra\v{c}ek \cite{MR2978140} for details.
\begin{proposition}\label{lem:subor}

The L\'{e}vy density $\nu$ in \eqref{eq:SBM} is smooth. If $\nu(r+1)\approx \nu(r)$, for $r\geq 1$, then
\begin{equation}\label{eq:nu-domination}
\left|\left(\frac{\rm d}{{\rm d}r}\right)^n\nu(r)\right|\leq C_n \nu(r), \quad r\ge 1, \; n\in \N,\end{equation}
in particular
{\rm \textbf{A1}} holds true.
\end{proposition}
\begin{proof}  Using \eqref{eq:SBM} we get, for $h>0$,
	\begin{equation}\label{eq:nugl}
\frac 1h\big(\nu(r+h) - \nu(r)\big) = \int_0^\infty g_t(r) \frac 1h \left(\me^{\frac {2rh-h^2}{4t}} - 1\right) \,\eta(\dt).
	\end{equation}
Let $0<h<r/4$.  Since $\me^{u}-1=\int_0^u \me^s\, \ds\leq u(1+{\me}^{u})$ for $u\geq 0$,  we get
\begin{eqnarray*}
0<\frac 1h \left(\me^{\frac{2rh-h^2}{4t}} - 1\right)&\leq& \frac{r}{2t}\left(1+ {\me}^{\frac{r^2}{8t}}\right)\end{eqnarray*}
and this quantity, multiplied by $g_t(r)$ is integrable with respect to $\eta.$
Letting $h\to 0$ in \eqref{eq:nugl} by the dominated convergence, we see that the
derivative of $\nu$ exists and
\begin{equation}\label{eq:partialnu}
 \nu'(r)= -\int_0^\infty\frac {r}{2 t}g_t(r) \, \eta(\dt), \quad r>0,
\end{equation}
so the integration and differentiation commute.
Continuity of the derivative is evident from \eqref{eq:partialnu}.
Higher order differentiability of $\nu$ can be established in the same way.

To prove \eqref{eq:nu-domination}
we first observe that for all $t>0$ and $r\geq 1$ we have $|g_t^{(n)}(r)|\leq W_n(r/2t) g_t(r),$ where $W_n$ is a  polynomial of  degree $n$ with nonnegative coefficients.
When $r\geq 2$ and $t\geq 0,$  then
$g_t(r)\left(r/2t\right)^n\leq  C_n  g_t(r-1),$
so that $|\nu^{(n)}(r)|\leq \widetilde C_n \nu(r-1)\approx \nu(r)$ for $r\ge 2,$ and thus for $r\geq 1$ because $\nu$ is strictly positive and decreasing.
\end{proof}

\begin{proposition}\label{lem:nuSc1}
If
$\int_r^\infty t^{-d/2}\,\eta(\dt)\geq c_1{e}^{-c_2r}$ for $r\geq 1$, then $\nu(r+1)\approx\nu(r)$, $r\geq1$.
\end{proposition}
\begin{proof}Assume that $\int_r^\infty t^{-d/2}\,\eta(\dt)\geq c_1{e}^{-c_2r}$ for $r\geq 1$.
By monotonicity of $\nu$, for $r\geq 1$ and $\lambda>0$,
$$\int^{r\lambda}_{0}g_t(r)\,\eta(\dt)\leq {e}^{-r^2/(8r\lambda)}\int^{r\lambda}_0g_t(r/\sqrt{2})\,\eta(\dt)\leq {e}^{-r/(8\lambda)}\nu(1/\sqrt{2})$$
and $$\int_{r}^{\infty}g_t(r)\,\eta(\dt)\geq (4\pi)^{-d/2}\me^{-r/4}
c_1{\me}^{-c_2 r}=c{e}^{-r(c_2+1/4)}.$$
Hence, for $\lambda_0=(8c_2+2)^{-1}$ we have
$$\int^{r\lambda_0}_{0}g_t(r)\,\eta(\dt)\leq c \int_{r}^{\infty}g_t(r)\,\eta(\dt),\quad r\geq 1.
$$
Since $\lambda_0<1$ we obtain
$$\nu(r)\approx \int^\infty_{r\lambda_0} g_t(r)\,\eta(\dt),\quad r\geq 1.$$
This yields
$$\nu(r+1)\geq \int^\infty_{0}g_t(r){e}^{-3r/(4t)}\,\eta(\dt)\geq {\me}^{-3/(4\lambda_0)} \int^\infty_{r\lambda_0}g_t(r)\,\eta(\dt)\approx\nu(r).$$
\end{proof}

\begin{lemma}
If $\varphi(\lambda)=\int_0^\infty(1-\me^{-t\lambda})\eta({\rm d}t)$ is complete Bernstein, then $\nu(r+1)\approx\nu(r)$, $r\geq 1$.
\end{lemma}
\begin{proof} By our assumptions $\eta({\rm d}t)=f(t){\rm d}t$
and there is a measure $\mu$ such that
$$f(t)=\int^\infty_0\me^{-ts}\,\mu(\ds),\quad s>0.$$
Let $s_0>0$ be such that $\mu((0,s_0])>0$. Then,
$$f(t)\geq  \me^{-ts_0}\mu((0,s_0]),\quad t>0,$$
hence for some constants $c_1,c_2$,
$$\int^\infty_rt^{-d/2}\,\eta(\dt)\geq c\int^\infty_rt^{-d/2}\me^{-s_0 t}\,\dt\geq c_1 {\me}^{c_2r},\quad r\geq 1.
$$
The lemma follows from Proposition \ref{lem:nuSc1}.
\end{proof}
By \cite[Theorem 5.18]{2016arXiv160604178G}, the inequality \eqref{eq:nuSc}  of {\rm \textbf{A2}} is satisfied if the derivative of $\varphi$ satisfies
\begin{equation}\label{eq:phiprim}
c^{-1}\varphi(r)\lambda^{-d/2-1+\beta}\leq \varphi'(\lambda r)\leq c \lambda^{-\alpha}\varphi(r),\quad \lambda,r\geq 1,
\end{equation}
for some $c,\alpha,\beta>0$. Next we discuss \eqref{es:nuSc1} in {\rm \textbf{A2}}.
The simplest situation arises when  inequalities \eqref{eq:phiprim} hold for every $r>0.$ Then \eqref{eq:nuSc} holds for every $r>0$ and therefore \eqref{es:nuSc1} holds as well. Hence the assumption {\rm \textbf{A2}} is satisfied in that case.
\begin{example}{\rm Assumptions {\rm \textbf{A1}} and {\rm \textbf{A2}} hold for the following operators:
\begin{align*}
L&=\Delta^{\alpha}\log^{\beta}(1+\Delta^{\gamma}) \quad \mbox{ if }
\gamma,\alpha,\alpha+2\beta\in [0,1), \ \gamma\beta +\alpha>0,\\
L&=\Delta^{\alpha_1}+\Delta^{\alpha_2}, \quad \mbox{ if }
\alpha_1,\alpha_2\in (0,1).
\end{align*}
The corresponding Bernstein functions and more examples are discussed in detail in \cite{MR2978140}}.
\end{example}

\appendix
\section{}
\subsection{
Not hitting the boundary}\label{app:A}
The boundary effects are easier to handle if the L\'evy process $X$ does not hit $\partial D$ at $\tau_D$. This motivates the following development.
Assume that $\nu$ satisfies {\rm \textbf{A2}}.
Then for every $R\in (0,\infty)$,
\begin{equation}\label{eq:nuScR}\nu(\lambda r)\leq c \lambda^{-d-\beta}\nu(r),\quad 0<\lambda\le 1,\, 0<r\leq R.
\end{equation}
Indeed, for $r\in (0,1]$ we can take $c=C$, and if $1<r\le R$, then
$$
\nu(\lambda r)\le \nu(\lambda 1) \le
C \lambda^{-d-\alpha} \nu(1)
\le C  \frac{\nu(1)}{\nu(R)}\ \lambda^{-d-\alpha}\nu(r).
$$
\noindent Let $K,h$ be the functions defined by \eqref{eq:Kdef}.
Recall that $K>0,$ $h>0$ and $h$ is strictly decreasing, but $r^2 h(r)$ is increasing.
Thus for $a\ge 1$ and
$r>0$,
\begin{equation}\label{e:ch}
h(r)\ge h(ar)=(ar)^2h(ar)/(ar)^2\ge r^2h(r)/(ar)^2=h(r)/a^2.
\end{equation}
Recall that $\omega_d=2 \pi^{d/2}/\Gamma(d/2)$ is the surface area of the unit sphere in $\Rd$. We obtain
\begin{eqnarray*}
	K(r)&=&r^{-2}\int_0^r \omega_d s^{d-1+2} \nu(s)\ds\ge \nu(r) r^{d}\omega_d/(d+2), \quad r>0.
\end{eqnarray*}
By \eqref{eq:nuScR}, for every $R<\infty$ we get
\begin{eqnarray*}
	K(r)&=&r^{-2}\int_0^r \omega_d s^{d+1} \nu(s)\ds
	\le c r^{-2}\int_0^r \omega_d s^{d+1} \nu(r)(s/r)^{-d-\beta}\ds\\
	&=& \nu(r)r^{d}c\omega_d/(2-\beta), \quad r\le R.
\end{eqnarray*}
Therefore,  for every $R\in (0,\infty)$,
\begin{equation}\label{e:cKnu}
\nu(r)\approx \frac{K(r)}{r^d},\quad 0<r\le R.
\end{equation}
If $0<r\le R/2$, then by \eqref{eq:nuScR} we have
\begin{eqnarray*}
	\nu(B^c_r)&=&\int_r^\infty \omega_d s^{d-1}\nu(s)\,\ds
	\ge c \int_r^R   s^{d-1}\frac{K(s)}{s^d}\,\ds\\
	&=&-c\int_r^R h'(s)\,\ds=c(h(r)-h(R))\ge c\left(1-\frac{h(R)}{h(R/2)}\right) h(r)\\
	&\ge& c(\nu(B_r^c)+K(r))\ge c \nu(B_r^c).
\end{eqnarray*}
Thus for every $R\in (0,\infty)$,
\begin{equation}\label{e:cnuh}
\nu(B^c_r)\approx h(r), \quad 0<r\le R/2.
\end{equation}

\begin{lemma}\label{l:nuSc}		
	If $|\partial D| = 0$ and VDC holds locally for $D^c$,
	then $\mathbb{P}^x(X_{\tau_D}\in \partial D) = 0$, $x\in D$.
\end{lemma}
For the narrower class of Lipschitz open sets and all isotropic pure-jump L\'evy processes with infinite L\'evy measure the result is stated by Sztonyk after Theorem~1 in \cite{MR1825650}. Our proof follows the argument given for the fractional Laplacian by Wu \cite[Theorem~1]{MR1893056}.

\begin{proof}[Proof of Lemma~\ref{l:nuSc}.]
	The trajectories of $X$ are c\`adl\`ag, so locally bounded, therefore
	$$\mathbb{P}^x(X_{\tau_D}\in \partial D)=\mathbb{P}^x(\tau_D<\infty, X_{\tau_D}\in \partial D)=\lim_{R\to \infty}\mathbb{P}^x(\tau_D<\tau_{B_R}, X_{\tau_D}\in \partial D).$$
	We have $\mathbb{P}^x(\tau_D<\tau_{B_R}, X_{\tau_D}\in \partial D)\leq \mathbb{P}^x(X_{\tau_{D\cap B_R}}\in \partial (D \cap B_R))$ for every $R>0$. Indeed, if $\tau_D<\tau_{B_R}$, then $X_{\tau_D}\in  B_R$ and it suffices to note that $\partial D\cap B_R\subset \partial (D\cap B_R)$. Therefore
	in what follows we may assume that $D$ is bounded and \eqref{eq:VDC} holds.
	Let $a=\max\{(2|B_1|/c)^{1/d},2\}$, where $c$ is the constant from \eqref{eq:VDC}.	
	By \eqref{e:cnuh},
	$$\nu(B^c_r)\approx h(r), \quad r\leq a^2\mathrm{diam}(D).$$
	Here, as usual,
	$$\mathrm{diam}(D)=\sup\{|x-y|: x,y\in D\}.$$
	For $x\in D$ we let $r_x = \delta_D(x)/2$ and $B_x = B(x,r_x)$. If $Q\in \partial D$ is such that $|x-Q| = \delta_D(x)$, then by \eqref{eq:VDC},
\begin{equation}\label{eq:Dc}
\left|D^c\cap\left(B(Q,ar)\setminus B(Q,r)\right)\right|\geq |B(Q,r)|,\quad r>0.
\end{equation}
	By unimodality of $\nu,$ \eqref{eq:Dc} and then \eqref{e:ch} we get
	\begin{align*}
	\nu(x,D^c)\geq&\sum_{k\geq 1} 	 \nu\left(D^c\cap\left(B(Q,a^kr_x)\setminus B(Q,a^{k-1}r_x)\right)-x\right)
	\\ \geq& \sum_{k\geq 1}  \nu(a^kr_x+2r_x)|B(Q,a^{k-1}r_x)|\\
\geq &\sum_{k\geq 1}  \nu(a^{k+1}r_x)|B(Q,a^{k+2}r_x)\setminus B(Q,a^{k+1}r_x)|\\
 =&\sum_{k\geq 1}  \nu(a^{k+1}r_x)|B(0,a^{k+2}r_x)\setminus B(0,a^{k+1}r_x)|\\
	\geq&  a^{-3d}\nu(B^c_{a^2r_x})\approx h(a^2r_x)\approx h(r_x).
	\end{align*}
	The estimates for Poisson kernel for the ball \cite[Lemma~2.2]{MR3729529} give
	$$P_{B_r}(0,z) \apprge \frac {\nu(z)}{h(r)}, \quad |z|>r>0.$$
	By \eqref{eq:IWG} we have $\omega_{B_x}(x,A):=\mathbb{P}^x(X_{\tau_{B_x}}\in A)=\int_{A} P_{B_{r_x}}(0,z-x)\,\dz$, if $\mathrm{dist}(D,A)>0$,
		hence
	\begin{align}\label{eq:poisest}
	\mathbb{P}^x(X_{\tau_{B_x}}\in D^c)
	\apprge \frac {\nu(x,D^c)}{h(r_x)} \geq c,
	\end{align}
	where $c>0$ does not depend on $x$.
	Following \cite{MR1893056}, we write
	\begin{align*}
	\mathbb P^x(X_{\tau_D}\in \partial D) = \mathbb{P}^x(X_{\tau_{B_x}}\in \partial D) + \mathbb{P}^x(X_{\tau_{B_x}}\in D, X_{\tau_D}\in\partial D)
	\end{align*}
	The first term vanishes because
	$|\partial D| = 0$. By the strong Markov property and \eqref{eq:poisest}, the second term is equal to
	\begin{align*}
	\int_{D\setminus B_x}\mathbb{P}^y(X_{\tau_D} \in \partial D)\,
	\omega_{B_x}(x,\dy) &\leq \sup\limits_{y\in D}\mathbb P^y(X_{\tau_D} \in \partial D)
	\mathbb{P}^x(X_{\tau_{B_x}}\in D\setminus B_x)\\
	&\leq (1-c) \sup\limits_{y\in D}\mathbb P^y(X_{\tau_D} \in \partial D).
	\end{align*}
	Thus, for every $x\in D$ we have
	\begin{equation*}
	\mathbb P^x (X_{\tau_D}\in \partial D) \leq (1-c)\sup\limits_{y\in D} \mathbb P^y(X_{\tau_D} \in \partial D).
	\end{equation*}
	This implies that
	$\sup\limits_{x\in D} \mathbb P^x(X_{\tau_D} \in \partial D)=0$.
	
\end{proof}
\begin{corollary}\label{c:achm}
	If $D$ is a bounded open set with $|\partial D| = 0$ and VDC holds for $D^c$ and $x\in D$, then
	\begin{align}\	\mathbb \mpr^x(X_{\tau_D} \in \mathrm{int}(D^c))&=\int_{D^c} P_D(x,y)\,\dy=1,
	\label{e:achm}\\
		\mathbb \mpr^x(X_{\tau_D-} \in D)&=1.\label{e:achm2}
	\end{align}
\end{corollary}
\begin{proof}
	Clearly, $\mpr^x(\tau_D<\infty)=1$, so \eqref{e:achm}  follows from Lemma~\ref{l:nuSc}. If $X_{\tau_D-}\in \partial D$ $a.s.$, then
	$X_{\tau_D}\in \partial D$ $a.s.$ \cite[proof of Lemma~17]{MR1438304}, proving \eqref{e:achm2}.
\end{proof}

\begin{corollary}\label{cor:ac}
	If  VDC holds locally for $D^c$,
	then
for nonnegative or integrable  $u$,
$$
\mathbb E^x u(X_{\tau_{D}})=\int_{D^c}u(y)P_{D}(x,y)\,\dy, \quad x\in D.
$$
\end{corollary}

\subsection{Approximation by smooth functions}\label{sec:c}
In this section we let $\nu$ be an arbitrary L\'evy measure on $\Rd$, i.e., we only assume that $\int_{\mR^d} (1\wedge|y|^2)\,\nu(\dy) < \infty$ and $\nu(\{0\}) = 0$.
In this general case the quadratic form is best defined as
$$\mathcal E _{D}(u,u) = \tfrac12\iint\limits_{\mR^d\times \mR^d} (u(x) - u(x+y))^2(\textbf{1}_{D}(x)\vee\textbf{1}_{D}(x+y)) \,\nu(\dy) \dx,$$
and $\vd_0$ is defined as before, cf. \cite{MR3738190}.
The following theorem is an extension of the result by Valdinoci et al. \cite{MR3310082}, where it was proved for the fractional Laplacian. The methods we use are rather classical, cf.
\cite[Section 5.3]{MR2597943}.
\begin{theorem}\label{th:density}
Let $\nu$ be an arbitrary L\'evy measure.
	If $D$ has continuous boundary and $u\in \vd_0$, then there are functions $\phi_n\in C^\infty_c(D)$ such that
	$
\ED(u-\phi_n,u-\phi_n) \rightarrow 0$ as $n\to \infty$.
\end{theorem}

We may construct the approximating functions $\phi_n$ in the same way as in \cite{MR3310082} provided that we check that the mollification, translation and cut-off are continuous in the seminorm $\sqrt{\ED(\cdot,\cdot)}$. We do this below.
Let $\eta\in C_c^\infty(B_1)$ be a nonnegative radial function on $\mR^d$ satisfying $\int_\Rd \eta(x)\, \dx = 1$ and let $\eta_\eps(x) = \eps^{-d}\eta(\frac x {\eps})$ for $\eps>0$, $x\in \Rd$. Here $B_r=B(0,r)$, as usual.

In the sequel we write $f_n \to f$ to denote $\lim_{n\to \infty}\ED(f-f_n,f-f_n) = 0$.

\begin{lemma}[Mollification]\label{lem:mol}
	For every $u\in\vd_0$, $\eta_\eps \ast u
	\to
	u$
	as $\eps \to 0$.
	\end{lemma}
\begin{proof}
	Note that $\mathcal E_{\mR^d}(u,u) = \ED(u,u) < \infty$.
	It suffices to verify that $\mathrm{I} := \iint\limits_{\mathbb R^d\times\mathbb R^d} (u\ast\eta_\eps(x) - u(x) - u\ast\eta_\eps(x+y) + u(x+y))^2 \,\nudy \dx \to 0$ as $\eps \to 0$. By Fubini--Tonelli theorem and Jensen's inequality,
	\begin{eqnarray}
	\mathrm{I} &=& \iint\limits_{\mathbb R^d\times\mathbb R^d} \left(\int_{B_1} \left(u(x-\eps z) - u(x+y - \eps z) - u(x) + u(x+y) \right)\eta(z)\, \dz \right)^2 \,\nudy \dx\nonumber\\
	&\leq& \int_{B_1}\int_{\mathbb R^d}\int_{\mathbb R^d}   \Big(u(x-\eps z) - u(x+y - \eps z) - u(x) + u(x+y)\Big)^2 \,\dx\nudy \,\eta (z) \dz\label{eq:moll}.
	\end{eqnarray}
	We will apply the dominated convergence theorem to the integral over $B_1\times \mR^d$.
	By the translation invariance of the Lebesgue measure,
	\begin{align*}
	&\eta(z)\pord (u(x-\eps z) - u(x+y - \eps z) - u(x) + u(x+y))^2\, \dx  \leq 4\eta(z)\pord (u(x) - u(x+y))^2\, \dx,
	\end{align*}
	which is integrable against $\nudy \dz$.
	Furthermore, by the continuity of translations in $L^2(\mR^d)$ the expression on the left-hand side converges to 0 as $\eps \to 0$, for every $z\in B_1,$ and $y\in\mR^d$. This ends the proof.
\end{proof}
We can use a similar reasoning to get the following fact.
\begin{corollary}[Translation]
	For every $u\in\vd_0$, $z\in B_1$, $u(\cdot + \eps z) \to
	u(\cdot)$
as $\epsilon \to 0$.
\end{corollary}
We can easily construct smooth functions $q_j$, $j\in\mathbb{N}$ such that $0\leq q_j \leq 1$, $q_j = 1$ in $B_j$, $q_j = 0$ in $B_{j+1}^c$ and such that $ |\nabla q_j(x)| < M$, $x\in\mR^d$, $j=1,2,\ldots$
\begin{lemma}[Cut-off]
	For every $u\in \vd_0$, $q_j u \to u$
as $j\to\infty$.
\end{lemma}
\begin{proof}
	Since $|(q_j u)(x) - (q_j u)(x+y) - u(x) + u(x+y)| \leq |(1-q_j(x))(u(x+y) - u(x))| + |(q_j(x) - q_j(x+y))u(x+y)|$, we get
	\begin{align}
	\mathcal{E}_{\Rd}(q_j u - u, q_j u - u)
	&\leq \pordd (1-q_j(x))^2(u(x) - u(x+y))^2\, \nudy \dx\label{s1}\\
	&+ \pordd (q_j(x) - q_j(x+y))^2u(x+y)^2 \,\nudy \dx.\label{s2}
	\end{align}
The integrands in \eqref{s1} and \eqref{s2} converge to $0$  $a.e.$ as $j\to \infty$.  For \eqref{s1} we have $(q_j(x) -1)^2 (u(x) - u(x+y))^2 \leq (u(x) - u(x+y))^2$,
which is integrable against $\nudy \dx$
since $u \in \vd_0$. For \eqref{s2} we use the smoothness of $q_j$:
	\begin{align*}
	(q_j(x) - q_j(x+y))^2 u(x+y)^2 \leq C(1\wedge|y|^2) u(x+y)^2,
	\end{align*}
	and so
	\begin{align*}
	\pord\pord (1\wedge|y|^2) u(x+y)^2 \,\dx\nudy  = \pord (1\wedge |y|^2)\, \nudy \pord u(x)^2 \,\dx < \infty.
	\end{align*}
By the dominated convergence theorem we obtain the desired result.
\end{proof}


\end{document}